\documentclass[11pt]{article}
\usepackage[utf8]{inputenc}
\usepackage[T1]{fontenc}
\usepackage{lmodern}
\usepackage[expansion=false]{microtype}
\usepackage{amsmath,amssymb,amsthm,mathtools}
\usepackage[margin=2.6cm]{geometry}
\usepackage{enumitem}
\usepackage{booktabs}
\usepackage{hyperref}

\theoremstyle{plain}
\newtheorem{theorem}{Theorem}[section]
\newtheorem{proposition}[theorem]{Proposition}
\newtheorem{lemma}[theorem]{Lemma}
\newtheorem{corollary}[theorem]{Corollary}
\theoremstyle{definition}
\newtheorem{definition}[theorem]{Definition}
\newtheorem{example}[theorem]{Example}
\newtheorem{problem}[theorem]{Problem}
\theoremstyle{remark}
\newtheorem{remark}[theorem]{Remark}

\newcommand{\Sy}{\mathcal S}
\newcommand{\ZR}{\mathbb Z[\mathbb R]}
\newcommand{\NR}{\mathbb N_0[\mathbb R]}
\newcommand{\cO}{\mathcal O}
\newcommand{\cF}{\mathcal F}
\newcommand{\cG}{\mathcal G}
\newcommand{\cT}{\mathcal T}
\newcommand{\cP}{\mathcal P}

\newcommand{\cE}{\mathcal E}
\newcommand{\Gr}{\operatorname{Gr}}

\title{The semiring of spectral duality skeletons:\\
cancellation, characters and factorization}
\author{J.~A.~Vega Coso\\
\small IUFFyM, Universidad de Salamanca\\
\small Salamanca, Spain\\
\small \href{mailto:jantovc@gmail.com}{jantovc@gmail.com}}
\date{}

\begin{document}
\maketitle

\begin{abstract}
In this paper we study weighted dual skeletons, which consist of a finite set with an involution and weights compatible with it, and which arise from spectral duality structures (SDS). Their isomorphism classes, under the Cartesian product and the disjoint union, form a graded commutative semiring with identity, which we identify completely: each class is determined by its degree, by a symmetric element $P$ of the group semiring $\mathbb N_0[\mathbb R]$ (with $\mathbb N_0=\{0,1,2,\dots\}$) and by its number $f$ of fixed points, and both operations are carried out componentwise. The multiplicative monoid is not cancellative, and the only obstruction is that multiplying by a skeleton without fixed points loses the information carried by $f$. The core of the paper is factorization. The reduced monoid is atomic with bounded lengths, but factorization is not unique, the elasticity is infinite, and not even the divisor-closed submonoid of skeletons with a single fixed point is half-factorial. For skeletons without fixed points, every reducible skeleton has a factor consisting of a single orbit exactly when the number of orbits is $8$, $16$, a prime or twice a prime, and in these cases this gives an explicit irreducibility criterion. Irreducibility is generic in almost all types, and progression skeletons factor according to mixed-radix number systems, with all their factorizations of the same length. Finally, we describe a continuous family of characters that, together with the degree and the number of fixed points, forms a complete invariant, the Grothendieck group and ring, and the interpretation of the semiring as an analogue of a crossed Burnside semiring.
\end{abstract}

\medskip
\noindent\textbf{MSC 2020:} 20M13, 20M14, 13A05, 16Y60, 19A22, 60J10.\quad
\textbf{Keywords:} commutative semiring; non-unique factorization; half-factoriality; Burnside ring; Grothendieck group; spectral duality structure.

\section{Introduction}\label{sec:intro}

This paper continues the algebraic study of spectral duality structures (SDS) and their weighted dual skeletons begun in \cite{PaperVI}. Let $\NR$ be the group semiring of $(\mathbb R,+)$ with nonnegative integer coefficients: its elements are the finite sums $P=\sum_tc_tx^t$, with $c_t\in\mathbb N_0=\{0,1,2,\dots\}$ and real exponents. We say that $P$ is \emph{symmetric} if the coefficient $c_t$ of $x^t$ equals the coefficient $c_{-t}$ of $x^{-t}$ for all $t$, that is, if $P(x^{-1})=P(x)$; for instance, $x^t+x^{-t}$ and $1+x^s+x^{-s}$ are symmetric, and $x^t$ with $t\ne0$ is not. We write $c_0(P)$ for its constant term. Let $\Sy=\{P\in\NR:\ P\ne0,\ c_t=c_{-t}\ \text{for all }t\}$ be the set of nonzero symmetric elements of $\NR$, and $\Sy^\sharp$ the set of pairs $(P,f)$ with $P\in\Sy$ and $f$ an integer such that
\[
0\le f\le c_0(P),\qquad f\equiv c_0(P)\pmod 2 .
\]
We define the operations
\[
(P,f)+(Q,g)=(P+Q,\,f+g),\qquad (P,f)(Q,g)=(PQ,\,fg).
\]
We show that $\Sy^\sharp$ is closed under these operations and that, together with the zero, which corresponds to the empty skeleton, it is a commutative semiring with identity. Our subject is its multiplicative arithmetic: which elements are irreducible, how the others factor, and to what extent factorization is unique.

The elements of $\Sy^\sharp$ are the isomorphism classes of \emph{weighted dual skeletons} of a fixed degree: finite sets with an involution and weights compatible with it (Section~\ref{sec:setting}). The number $f$ is the number of fixed points of the involution, and $P=\sum_zx^{\log r_z}$ records the weights through the quotients $r_z$. Fixed points have $r_z=1$, and orbits with $r_z=1$ contribute $2$ to the constant term; hence the two conditions on $f$. The Cartesian product of skeletons corresponds to the product of pairs, and the disjoint union of skeletons of the same degree to the sum (Theorem~\ref{thm:semiring}). We work with skeletons rather than complete SDS because the product of two complete structures, although it is again a complete SDS, loses in general the linear independence of its spectral functions, which is the hypothesis of the rank theorem (Remark~\ref{rem:complete}); the classes of skeletons, on the other hand, form a semiring graded by the degree (Definition~\ref{def:graded} and Proposition~\ref{prop:graded}). The degree $K>0$ is multiplicative. Since the one-point skeletons are the units of the semiring, one for each degree, the degree plays no role in factorization, which takes place in $\Sy^\sharp$.

An example shows the central phenomenon. Consider three skeletons: the one formed by a single orbit of ratio $t\ge0$, whose pair is $\cO(t)=(x^t+x^{-t},0)$; the one formed by two fixed points and no orbit, with pair $\cF_2=(2,2)$; and the one formed by a single orbit of ratio $0$, with pair $\cO(0)=(2,0)$. Then
\[
\cO(t)\,\cF_2=(2x^t+2x^{-t},\,0)=\cO(t)\,\cO(0),
\]
but $\cF_2\ne\cO(0)$: both have $P=2$, but the first has two fixed points and the second has none. So the multiplicative monoid is not cancellative: multiplying by an element with $f=0$ loses the information about the fixed points, and we prove that this is the only obstruction to cancellation (Theorem~\ref{thm:nocancel}). As a consequence, the Grothendieck ring, which we describe explicitly (Theorem~\ref{thm:ring}), has zero divisors.

\paragraph{Results on factorization.}
The reduced monoid is atomic and its factorization lengths are bounded (Proposition~\ref{prop:atomic}). However:
\begin{itemize}[nosep]
\item factorization is not unique (Theorem~\ref{thm:nonunique}), the monoid is not half-factorial (Theorem~\ref{thm:halffact}), and its elasticity is infinite (Theorem~\ref{thm:elast}). In the last two, the mechanism is, once again, the loss of the fixed points;
\item the divisor-closed submonoid of skeletons with $f=1$, where this mechanism does not act, is not half-factorial either (Theorem~\ref{thm:f1});
\item when $f=0$ and the number $k\ge2$ of orbits is $8$, $16$, a prime or twice a prime, every reducible skeleton has a factor consisting of a single orbit, and this gives an explicit criterion for irreducibility in terms of \emph{centres}. For every other $k\ge2$ there are reducible skeletons with no such factor, and the smallest such $k$ is $9$ (Theorem~\ref{thm:orbitfactor}, Example~\ref{ex:k9}). In general, reducibility is decided by a finite search (Proposition~\ref{prop:general});
\item irreducibility is generic in almost all types (Theorem~\ref{thm:generic});
\item \emph{progression skeletons} (Example~\ref{ex:ruin}), whose exponents form a symmetric arithmetic progression and which include the skeleton of the gambler's ruin, where the theory began, have symmetric $q$-integers as their $\Phi$. Their factorizations are exactly de~Bruijn's mixed-radix number systems, and all factorizations into irreducibles have the same length (Theorem~\ref{thm:ruin}).
\end{itemize}
We also describe a continuous family of characters of the semiring which, together with the degree and $f$, forms a complete invariant (Proposition~\ref{prop:chi}), and we study the multiplicative Grothendieck group, which in the commensurable case is determined by Pólya's theorem (Theorem~\ref{thm:polya}).

\paragraph{Related work.}
For $\mathbb N_0[x]$, the multiplicative monoid of its nonzero elements is not half-factorial and has infinite elasticity \cite{CCMS2009} (see also \cite{CampaniniFacchini}). More generally, Chapman and Polo \cite{ChapmanPolo} study the multiplicative arithmetic of semidomains of the form $S[M]$, with $S$ a semidomain and $M$ a torsion-free monoid: they characterize bounded, finite and unique factorization, and find large classes with infinite elasticity. Gotti and Polo \cite{GottiPolo} study polynomial semidomains. In this language, $\Sy$ is the multiplicative submonoid of symmetric elements of the monoid semidomain $\NR$. The results of this paper that involve only $P$ belong to that circle of questions, with two differences: the exponents are real and the elements are symmetric. Those that involve $f$ (the failure of cancellation, the zero divisors of the Grothendieck ring and the mechanism behind the infinite elasticity) have no counterpart there. We use the terminology of factorization theory from \cite{GHK}.

Two other lines of work are close to ours:
\begin{enumerate}
\item The theory of Burnside rings. Skeletons of fixed degree are exactly the finite $C_2$-sets over the $C_2$-monoid $(\mathbb R,+)$, with $C_2$ acting by $t\mapsto-t$ (Proposition~\ref{prop:burnside}). The pair $(P,f)$ consists of their marks, the parity condition is a Dress congruence, and the failure of cancellation is the classical zero divisor of the Burnside ring of $C_2$. The semiring is thus an analogue, for an infinite monoid, of the semirings of crossed $G$-sets whose Grothendieck rings are the crossed Burnside rings of Oda and Yoshida \cite{OdaYoshida} (Section~\ref{subsec:burnside}). The additive and ring structures follow from this theory; the contribution of this paper is the multiplicative arithmetic of the semiring.
\item The factorization of finite graphs. Van de Woestijne \cite{vdWoestijne2012} studies, through polynomials with nonnegative coefficients, how uniqueness of factorization depends on the number of components, under additional hypotheses. Imrich, Klep and Smertnig \cite{ImrichKlepSmertnig} embed graph semirings in monoid algebras, observe that $\mathbb N_0[Y]$ does not have unique factorization, and propose the study of its non-unique factorizations \cite[Remark~3.17]{ImrichKlepSmertnig}. The failure of cancellation is also reminiscent of the bipartite exceptions to Lovász's cancellation theorem for the direct product of graphs \cite{Lovasz1971,Hammack2008}.
\end{enumerate}

\paragraph{Origin.}
Skeletons are the combinatorial part of SDS, which arose to explain why certain distributions make the absorption probabilities of resetting processes independent of the resetting mechanism, with a universal value $C^\ast=1/(1+\sqrt K)$ \cite{PaperI,PaperII,PaperIII,PaperIV,PaperV}. The algebra of these structures (operations, morphisms and classification) is studied in \cite{PaperVI}, and their axioms and the internal rigidity of a single structure in \cite{PaperVII}. This paper does not need that motivation, but we return to it in Section~\ref{sec:burnside}. There, the number of $C_2$-orbits and the number of free orbits, which are not marks but are computed from them by the orbit-counting lemma, give the dynamical invariants of the structures: the dimension of the separatrix of neutral distributions and the response rank. Moreover, neutral distributions of the factors give, by tensor product, neutral distributions of the product, with an explicit codimension. Which structures are realized by positive Markov processes, and whether the realizable classes form a subsemiring, is a question for the line of work on realizability and is not treated here.

\paragraph{Organization.}
Section~\ref{sec:setting} defines skeletons and the semiring; Section~\ref{sec:cancel} deals with cancellation, characters and the Grothendieck group and ring; Section~\ref{sec:fact} contains the factorization results, and Section~\ref{subsec:ruin} the progression skeletons. Section~\ref{sec:burnside} explains the Burnside viewpoint, orbit counting and the dynamical origin, and Section~\ref{sec:open} collects open problems.

\section{Skeletons and the semiring of classes}\label{sec:setting}

In this section we fix the object of study and build the tool we will use in the rest of the paper. We first recall what a skeleton is and how two skeletons are multiplied and added. We then attach to each skeleton a pair $(\Phi,f)$, formed by a symmetric polynomial $\Phi$ with real exponents and the number $f$ of fixed points of the involution. The main result of the section, Theorem~\ref{thm:semiring}, says that this pair, together with the degree, determines the skeleton, and that operations on skeletons become operations on pairs. From then on, questions about skeletons become questions about polynomials, and this is why the paper is about arithmetic. We also fix the notation for the examples, organize the classes into a graded semiring, and end with the progression skeletons, a family of examples that will reappear throughout the paper.

\subsection{Skeletons}

We follow the definition of \cite[Definition~2.1]{PaperVI}, which corresponds to axioms (A1)--(A4) of \cite{PaperVII}, with one difference: we allow $Z$ to be empty. A \emph{weighted dual skeleton of degree $K>0$}, denoted $\cE=(Z,\sigma,\kappa)$, consists of a finite set $Z$, possibly empty, whose elements we call \emph{sites}, an involution $\sigma$ of $Z$, and a weight $\kappa:Z\to(0,\infty)$ such that
\[
\kappa(z)\,\kappa(\sigma z)=K\qquad\text{for all }z\in Z .
\]
If $Z$ is not empty, the degree is determined by $\kappa$, since $K=\kappa(z)\kappa(\sigma z)$ for any site $z$. If $Z$ is empty, $\kappa$ is the empty map and the condition holds vacuously for every $K$; so in that case the degree has to be given separately, as part of the definition, and there is an \emph{empty skeleton} $\emptyset_K$ for each $K>0$. Two skeletons are \emph{isomorphic} if they have the same degree and there is a bijection between their sites that commutes with the involutions and preserves the weights \cite[Definition~7.1]{PaperVI}. When $Z$ is not empty, equality of degrees is automatic.

Condition (A1) of \cite{PaperVII} asks for $Z$ to be finite and non-empty. We keep finiteness and drop non-emptiness. The empty skeleton satisfies the other conditions vacuously, so non-emptiness only excludes this case. Admitting it changes nothing for non-empty skeletons, and it gives the zero of the semiring directly (Section~\ref{subsec:graded}).

We write $m=|Z|$ for the number of sites. The \emph{orbit} of a site $z$ is the set of sites obtained from $z$ by applying $\sigma$; since $\sigma^2$ is the identity, it is simply $\{z,\sigma z\}$, and there are only two possibilities. If $\sigma z=z$, the orbit is $\{z\}$, and we say that $z$ is a \emph{fixed point} (a \emph{spectrally neutral site} in \cite{PaperII}). If $\sigma z\ne z$, the orbit is a pair $\{z,\sigma z\}$ of two distinct sites interchanged by $\sigma$ (a \emph{spectral pair} in \cite{PaperII}). The orbits form a partition of $Z$, so that, if $f$ is the number of fixed points and $k$ the number of pairs, then $m=2k+f$.

\emph{Convention.} Except in Section~\ref{sec:burnside}, we reserve the word \emph{orbit} for the pairs $\{z,\sigma z\}$ with $z\ne\sigma z$, and count the fixed points separately: a skeleton has $f$ fixed points and $k$ orbits. In the language of Section~\ref{sec:burnside}, where $\sigma$ is seen as an action of the group $C_2$, these orbits are the free orbits, and the fixed points are the one-element orbits.

The weights $\kappa(z)$ and $\kappa(\sigma z)$ of the two sites of an orbit are compared through the quotient
\[
r_z=\Bigl(\frac{\kappa(\sigma z)}{\kappa(z)}\Bigr)^{1/2},\qquad\text{so that}\qquad r_{\sigma z}=r_z^{-1}\quad\text{and}\quad r_z=1\ \text{at fixed points}.
\]
In \cite[Definition~2.7]{PaperVI} and in \cite{PaperVII} this quotient is called the \emph{orbit ratio}. Here the \emph{ratio} of an orbit $O=\{z,\sigma z\}$ is $t_O=|\log r_z|\ge0$, which does not depend on which site of the orbit is chosen. An orbit of ratio $t$ will contribute $x^t+x^{-t}$ to the polynomials with which we describe skeletons, and a balanced orbit has ratio $0$ in this paper and ratio $1$ in \cite{PaperVI,PaperVII}. Note that an orbit of ratio $0$ has equal weights but is \emph{not} a fixed point: it still consists of two distinct sites interchanged by $\sigma$. This difference between a fixed point and a balanced orbit will be essential in Section~\ref{sec:cancel}.

The data just introduced are enough to describe a skeleton up to isomorphism. The following result, proved in \cite{PaperVI}, is the basis of everything that follows.

\begin{theorem}[Classification {\cite[Theorem~7.2]{PaperVI}}]\label{thm:class}
The isomorphism class of $\cE$ is determined by its degree $K$, its number of fixed points ($f$) and the multiset of the ratios of its $k$ orbits ($\{t_O\}$). Conversely, for every $K>0$, every $f\ge0$ and every multiset of $k$ nonnegative real numbers there is a skeleton with these data; for $f=k=0$ it is the empty skeleton $\emptyset_K$.
\end{theorem}

\begin{proof}
Since $\kappa(z)\kappa(\sigma z)=K$ and $r_z^2=\kappa(\sigma z)/\kappa(z)$, we have $\kappa(z)=\sqrt K\,r_z^{-1}$: the weights are determined by $K$ and by the values $\log r_z$. If two skeletons have the same $K$, the same $f$ and the same ratios with multiplicity, a bijection between their fixed points, together with a bijection between orbits of the same ratio that sends the site with $\log r_z\ge0$ to the corresponding site, commutes with the involutions and preserves $\log r$, hence preserves $\kappa$. Conversely, the skeleton with $f$ fixed points of weight $\sqrt K$ and, for each $t$ in the multiset, an orbit with weights $\sqrt K\,e^{\mp t}$ has the given data.
\end{proof}

\paragraph{Operations.}
There are two natural ways of building new skeletons from two given ones \cite[Section~4]{PaperVI}. The \emph{product} $\cE_1\times\cE_2$ has as sites the pairs $(z,w)\in Z_1\times Z_2$, as involution $\sigma_1\times\sigma_2$, and as weight $\kappa_1\otimes\kappa_2$, that is, $(z,w)\mapsto\kappa_1(z)\kappa_2(w)$. It is a skeleton of degree $K_1K_2$, and its quotients multiply: $r_{(z,w)}=r_zr_w$ \cite[Proposition~4.2]{PaperVI}.

The second operation only makes sense for skeletons of the same degree. If $\cE_1$ and $\cE_2$ both have degree $K$, their \emph{disjoint union} $\cE_1\sqcup\cE_2$ has as sites $Z_1\sqcup Z_2$, with the involution and the weight of each part, and it is again a skeleton of degree $K$. If the degrees were different, the product $\kappa(z)\kappa(\sigma z)$ would be $K_1$ at the sites of $\cE_1$ and $K_2$ at those of $\cE_2$, and there would be no common constant.

Both operations are compatible with isomorphism, and they are associative and commutative up to isomorphism. The one-point skeleton with $\kappa=K=1$ is the neutral element for the product.

\subsection{The encoding}

To study the classes of skeletons under these operations, we encode them by an algebraic object in which the product of skeletons becomes an ordinary product. The idea is to record, for each site, the value of $\log r_z$, and to do it with exponents, so that they add up under multiplication.

Let $\ZR$ be the group ring of $(\mathbb R,+)$: its elements are the finite sums $P=\sum_tc_tx^t$, with coefficients $c_t\in\mathbb Z$ and exponents $t\in\mathbb R$, which are added and multiplied like polynomials, with the rule $x^sx^t=x^{s+t}$. Let $\NR$ be the subsemiring of elements with coefficients in $\mathbb N_0$. For an arbitrary element $P$ of $\ZR$, not necessarily symmetric nor with non-negative coefficients, we write $c_t(P)$ for the coefficient of $x^t$; in particular, $c_0(P)$ is its constant term. We say that $P$ is \emph{symmetric} if $c_t(P)=c_{-t}(P)$ for all $t$, and we write $\Sy\subset\NR$ for the set of nonzero symmetric elements.

To each skeleton we attach
\[
\Phi(\cE)=\sum_{z\in Z}x^{\log r_z}=f+\sum_O\bigl(x^{t_O}+x^{-t_O}\bigr)\in\NR .
\]
The expression $f+\sum_O(x^{t_O}+x^{-t_O})$ comes from grouping the sites: each fixed point contributes $x^0=1$, and each orbit of ratio $t_O$ contributes $x^{t_O}+x^{-t_O}$, because its two sites have opposite values of $\log r$. This is why $\Phi(\cE)$ is symmetric. If $Z$ is not empty, it is also nonzero, that is, $\Phi(\cE)\in\Sy$. The empty skeleton has $\Phi=0$ and $f=0$.

We fix notation for a few simple skeletons, which we use as examples throughout the paper.

\paragraph{Blocks.}
We denote by $\cO(t)$ the skeleton consisting of a single orbit of ratio $t$, that is, two sites interchanged by $\sigma$ and no fixed point, with $\Phi=x^t+x^{-t}$, $m=2$ and $f=0$; by $\cF_n$ the skeleton consisting of $n$ fixed points and no orbit, with $\Phi=n$ and $m=f=n$; and by $\cT(s)$ the skeleton consisting of one fixed point and one orbit of ratio $s$, with $\Phi=1+x^s+x^{-s}$, $m=3$ and $f=1$. The degree of these blocks is arbitrary, and we omit it when it plays no role.

The encoding $\Phi$ is compatible with both operations, and so are the number of fixed points and the number of sites. Observe that $m$ is obtained from $\Phi$ by evaluating at $x=1$: each site contributes $1^{\log r_z}=1$, so that $\Phi(\cE)(1)=m$.

\begin{proposition}[Additivity and multiplicativity]\label{prop:mult}
$\Phi$ is additive for the disjoint union and multiplicative for the product:
\[
\Phi(\cE_1\sqcup\cE_2)=\Phi(\cE_1)+\Phi(\cE_2),\qquad
\Phi(\cE_1\times\cE_2)=\Phi(\cE_1)\,\Phi(\cE_2).
\]
The same holds for the number of fixed points and for the number of sites: $f(\cE_1\sqcup\cE_2)=f_1+f_2$ and $f(\cE_1\times\cE_2)=f_1f_2$, and similarly for $m$.
\end{proposition}

\begin{proof}
For the disjoint union, the sum defining $\Phi$ splits into the sum over $Z_1$ and the sum over $Z_2$. For the product, since $r_{(z,w)}=r_zr_w$, we have $\log r_{(z,w)}=\log r_z+\log r_w$, and therefore
\[
\Phi(\cE_1\times\cE_2)=\sum_{(z,w)}x^{\log r_z+\log r_w}=\Bigl(\sum_{z}x^{\log r_z}\Bigr)\Bigl(\sum_{w}x^{\log r_w}\Bigr)=\Phi(\cE_1)\,\Phi(\cE_2).
\]
A pair $(z,w)$ is fixed by $\sigma_1\times\sigma_2$ if and only if both $z$ and $w$ are fixed, so $f(\cE_1\times\cE_2)=f_1f_2$; in the disjoint union the fixed points simply add up. For $m$ it is enough to evaluate the above identities at $x=1$.
\end{proof}

In particular, the proposition explains what happens to orbits under multiplication. Two orbits of ratios $t_1$ and $t_2$ give a skeleton with four sites and no fixed points, that is, with two orbits, and their ratios can be read off from the product of the two binomials.

\begin{remark}[Splitting of orbits]\label{rem:splitting}
We have
\[
(x^{t_1}+x^{-t_1})(x^{t_2}+x^{-t_2})=(x^{t_1+t_2}+x^{-t_1-t_2})+(x^{|t_1-t_2|}+x^{-|t_1-t_2|}),
\]
that is, $\cO(t_1)\times\cO(t_2)$ has two orbits, of ratios $t_1+t_2$ and $|t_1-t_2|$. This splitting of orbits in the product, observed in \cite[Remark~4.7]{PaperVI}, is therefore just an identity in the group ring. It is the mechanism behind almost all the factorizations of Section~\ref{sec:fact}.
\end{remark}

The polynomial $\Phi$, however, is not enough to describe a skeleton. Its constant term counts both the fixed points and the orbits of ratio $0$, since each of the latter contributes $x^0+x^{-0}=2$:
\[
c_0(\Phi)=f+2n_0,\qquad n_0=\text{number of orbits of ratio }0.
\]
Thus $\cF_2$ and $\cO(0)$ have the same $\Phi=2$ and are not isomorphic. This is why we record $f$ separately, as a second component. The identity $c_0(\Phi)=f+2n_0$ also says that $f\le c_0(\Phi)$ and that $f$ and $c_0(\Phi)$ have the same parity. These two conditions are the only restrictions, and they lead to the following definition.

\begin{definition}
Let
\[
\Sy^\sharp=\bigl\{(P,f)\in\Sy\times\mathbb N_0:\ f\le c_0(P),\ f\equiv c_0(P)\pmod 2\bigr\},
\]
with componentwise sum and product $(P,f)(Q,g)=(PQ,fg)$.
\end{definition}

For this definition to be useful, we have to check that the operations do not lead out of $\Sy^\sharp$, that is, that they are internal. For the sum this is clear; for the product we need to look at the constant term of a product.

\begin{lemma}\label{lem:closed}
$\Sy^\sharp$ is closed under sum and product, and the product is distributive over the sum. Consequently, $\Sy^\sharp_0=\Sy^\sharp\cup\{(0,0)\}$ is a commutative subsemiring with identity of $\NR\times\mathbb N_0$, with zero $(0,0)$ and identity $(1,1)$.
\end{lemma}

\begin{proof}
For the sum, it is enough to add the inequalities and the congruences. For the product, let $(P,f),(Q,g)\in\Sy^\sharp$. The constant term of $PQ$ collects the products of terms with opposite exponents, $c_t(P)\,c_{-t}(Q)$; since $Q$ is symmetric, $c_{-t}(Q)=c_t(Q)$, and grouping $t$ with $-t$ gives
\[
c_0(PQ)=c_0(P)\,c_0(Q)+2\sum_{t>0}c_t(P)\,c_t(Q).
\]
This quantity is $\ge c_0(P)\,c_0(Q)\ge fg$, and it is congruent to $c_0(P)\,c_0(Q)$, hence to $fg$, modulo $2$. Moreover, $PQ$ is symmetric, because the reflection $x^t\mapsto x^{-t}$ is a ring homomorphism fixing $P$ and $Q$, and it is nonzero, because if $s$ and $t$ are the largest exponents of $P$ and $Q$, the coefficient of $x^{s+t}$ in $PQ$ is $c_s(P)c_t(Q)\ne0$. Distributivity is checked componentwise.
\end{proof}

We can now state the main result of the section. It says that $\Sy^\sharp$, together with the degree, describes exactly the classes of skeletons, with no loss of information and no superfluous elements, and that the operations on skeletons are those of $\Sy^\sharp$.

\begin{theorem}[Identification]\label{thm:semiring}
The map $[\cE]\mapsto(K,\Phi(\cE),f)$ is a bijection from the set of isomorphism classes of skeletons onto $\mathbb R_{>0}\times\Sy^\sharp_0$, which takes the empty skeletons to $\mathbb R_{>0}\times\{(0,0)\}$ and the non-empty ones to $\mathbb R_{>0}\times\Sy^\sharp$. It takes the product of skeletons to the product
\begin{equation}\label{eq:prodclases}
(K,P,f)(K',Q,g)=(KK',PQ,fg)
\end{equation}
and, for skeletons of the same degree, the disjoint union to the sum of $\Sy^\sharp$. In particular, if $\cE_1$ and $\cE_2$ have the same degree and $\cE$ is any skeleton,
\begin{equation}\label{eq:distrib}
\cE\times(\cE_1\sqcup\cE_2)\ \cong\ (\cE\times\cE_1)\sqcup(\cE\times\cE_2).
\end{equation}
\end{theorem}

\begin{proof}
\emph{Injectivity.} By the identity $c_0(\Phi)=f+2n_0$, the pair $(\Phi,f)$ determines $n_0$; and, for $t>0$, the coefficient of $x^t$ in $\Phi$ is the number of orbits of ratio $t$. So $(K,\Phi,f)$ determines $(K,f,\{t_O\})$, which determines the class by Theorem~\ref{thm:class}.

\emph{Surjectivity.} Given $(K,P,f)$ with $(P,f)\in\Sy^\sharp_0$, the two conditions defining $\Sy^\sharp$ say precisely that $n_0=(c_0(P)-f)/2$ is a nonnegative integer. We then take $f$ fixed points, $n_0$ orbits of ratio $0$ and $c_t(P)$ orbits of ratio $t$ for each $t>0$, all of degree $K$; by Theorem~\ref{thm:class} this skeleton exists, and its image is $(K,P,f)$. For $(P,f)=(0,0)$ it is the empty skeleton $\emptyset_K$.

\emph{Operations.} Formula~\eqref{eq:prodclases} and the corresponding one for the disjoint union are Proposition~\ref{prop:mult}, together with the multiplicativity of the degree. For~\eqref{eq:distrib}, both sides have the same image by Lemma~\ref{lem:closed} (distributivity in $\Sy^\sharp$), and by injectivity they are isomorphic.
\end{proof}

From now on we identify each isomorphism class with its triple $(K,\Phi,f)$ and, when the degree plays no role, with its pair $(\Phi,f)$.

\subsection{Units and the graded semiring}\label{subsec:graded}

Theorem~\ref{thm:semiring} leaves two questions to settle before we study factorization. The first is which elements are invertible for the product, because a factorization only makes sense up to such elements. The second is that only skeletons of the same degree can be added, so that the classes of skeletons, as they stand, do not form a semiring. We deal with them in turn.

\begin{remark}[Units]\label{rem:units}
The neutral element for the product is the class of the one-point skeleton of degree $1$, that is, $Z=\{z\}$, $\sigma=\mathrm{id}$ and $\kappa(z)=1$ (the degree is $\kappa(z)\kappa(\sigma z)=\kappa(z)^2=1$), which corresponds to the triple $(K,\Phi,f)=(1,1,1)$. A class is invertible (a \emph{unit}) if and only if it has $m=1$: indeed, $m$ is multiplicative and takes integer values $\ge0$, so $m(\cE)\,m(\cE')=1$ forces $m(\cE)=1$; in particular, the empty skeleton is not invertible. The classes with $m=1$ are those of the one-point skeletons, which correspond to the triples $(K,1,1)$, one for each degree $K>0$, and they are invertible, since $(K,1,1)(K^{-1},1,1)=(1,1,1)$. The group of units is therefore $\{(K,1,1):K>0\}\cong\mathbb R_{>0}$, and a unit carries no information other than its degree.

Consequently, multiplying by a unit only changes the degree, and two classes obtained from each other by multiplying by a unit have the same $(\Phi,f)$. Identifying each class with those obtained by multiplying it by a unit (that is, passing to the \emph{reduced monoid}) forgets exactly the degree, and what remains is $\Sy^\sharp$ with its product. This is why the factorization of Section~\ref{sec:fact} takes place in $\Sy^\sharp$, and the degree plays no role in it.

\emph{Convention.} The empty skeleton is absorbing for the product and plays no role in the multiplicative arithmetic. For this reason, in the rest of the paper and unless stated otherwise, \emph{skeleton} means non-empty skeleton, and the multiplicative monoid is that of the classes of non-empty skeletons, whose reduced monoid is $\Sy^\sharp$. This is what one does with the zero of a domain.
\end{remark}

For the sum, the empty skeleton $\emptyset_K$ is the neutral element for the disjoint union among skeletons of degree $K$. But the disjoint union of skeletons of different degrees is not a skeleton. Besides, the additive structure is very poor. The orbits of $\sigma$ partition $Z$, and each of them, with its weights, is itself a skeleton of degree $K$: a fixed point is a copy of $\cF_1$, and an orbit of ratio $t$ a copy of $\cO(t)$, both of degree $K$ (Theorem~\ref{thm:class}). Thus every skeleton of degree $K$ is the disjoint union of $f$ copies of $\cF_1$ and one copy of $\cO(t_O)$ for each orbit, and this decomposition is unique up to order, because $f$ and the multiset of ratios are invariant under isomorphism. Adding skeletons therefore simply means putting these pieces together: the classes of degree $K$ form the free commutative monoid generated by $\cF_1$ and the $\cO(t)$, $t\ge0$, with neutral element $\emptyset_K$. All the interesting arithmetic is multiplicative. Even so, it is convenient to have a genuine semiring, and the standard way to obtain one is to add classes of different degrees formally.

\begin{definition}[The graded semiring of classes]\label{def:graded}
For $K>0$, let $\mathfrak S_K$ be the set of isomorphism classes of skeletons of degree $K$, including the empty skeleton $\emptyset_K$, which is neutral for the disjoint union. Let
\[
\mathfrak S=\bigoplus_{K>0}\mathfrak S_K
\]
be the set of families $(a_K)_{K>0}$, with $a_K\in\mathfrak S_K$ and $a_K=\emptyset_K$ for all but finitely many degrees, with componentwise sum and with product
\[
(ab)_K=\bigsqcup_{K_1K_2=K}a_{K_1}\times b_{K_2},
\]
where the product with an empty skeleton is empty. Each non-empty class $a\in\mathfrak S_K$ is identified with the element of $\mathfrak S$ that has $a$ in degree $K$ and the empty skeleton in all other degrees, that is, with the family $(a_{K'})_{K'>0}$ given by $a_K=a$ and $a_{K'}=\emptyset_{K'}$ for $K'\ne K$; these are the nonzero \emph{homogeneous} classes of $\mathfrak S$. The zero of $\mathfrak S$ is the family $(\emptyset_K)_{K>0}$.
\end{definition}

\begin{remark}
In \cite[Remark~5.6]{PaperVI} the empty skeleton appeared as a formal object. Here it is one more skeleton, and the zero of $\mathfrak S$ needs no additional construction: it is the family of empty skeletons.
\end{remark}

With this construction, Theorem~\ref{thm:semiring} can be restated as an isomorphism of semirings: the degree becomes a grading, and the component of each degree is $\Sy^\sharp_0$.

\begin{proposition}\label{prop:graded}
$\mathfrak S$ is a commutative semiring with identity, graded by the multiplicative group $\mathbb R_{>0}$. Its zero is the family of empty skeletons, and its identity is the class of the one-point skeleton of degree $1$. Theorem~\ref{thm:semiring} gives an isomorphism of graded semirings
\[
\mathfrak S\cong\Sy^\sharp_0[\mathbb R_{>0}],
\]
the group semiring, with coefficients in $\Sy^\sharp_0$, of the group of degrees, whose elements are the finite sums $\sum_Ka_K\varepsilon_K$, with $a_K\in\Sy^\sharp_0$, where $\varepsilon_K$ denotes the degree $K$, with the rule $\varepsilon_K\varepsilon_{K'}=\varepsilon_{KK'}$.
\end{proposition}

\begin{proof}
Each $\mathfrak S_K$ is a commutative monoid under disjoint union. The product is well defined on classes, and it is associative, commutative and distributive over the sum by Theorem~\ref{thm:semiring}, since only classes of the same degree are added. The isomorphism takes the class of a skeleton of degree $K$ to $(\Phi,f)\cdot\varepsilon_K$; it is additive and multiplicative by Proposition~\ref{prop:mult} and the multiplicativity of the degree.
\end{proof}

Throughout the paper, multiplicative arithmetic refers to homogeneous elements, that is, to classes of skeletons. We only add classes of the same degree, apart from the formal sums in $\mathfrak S$.

We end the subsection by placing $\mathfrak S$ with respect to the usual notions of semiring theory, which we will use when comparing with the literature.

\begin{remark}[Further properties]\label{rem:props}
In the terminology of \cite{Golan1999}, the semiring $\mathfrak S$ is \emph{zerosumfree}: a sum of nonzero elements is never zero, because adding adds sites. It is also \emph{entire}, that is, it has no zero divisors: if $a,b\ne0$ and $K_a$, $K_b$ are the largest degrees at which they do not vanish, then $(ab)_{K_aK_b}=a_{K_a}\times b_{K_b}$, which has $m(a_{K_a})\,m(b_{K_b})\ge1$ sites, and hence $ab\ne0$.

In a ring, having no zero divisors is equivalent to being able to cancel nonzero factors. In a semiring this is not so, and $\mathfrak S$ is an example: it has no zero divisors, but it is not multiplicatively cancellative (Theorem~\ref{thm:nocancel}). As a consequence, neither $\mathfrak S$ nor its degree-one part $\Sy^\sharp_0$ is a \emph{semidomain} in the sense of \cite{GottiPolo,ChapmanPolo}, that is, a subsemiring of an integral domain, since semidomains are cancellative. The component $\Phi$ alone takes values in the semidomain $\NR$; the obstruction comes exactly from the component $f$.
\end{remark}

\subsection{A family of examples: progression skeletons}

We close the section with a family of examples that will reappear several times and that connects with the origin of the theory.

\begin{example}[Progression skeletons]\label{ex:ruin}
Let $n\ge1$ and $\beta>0$. We call \emph{progression skeleton}, and denote by $\cP_n(\beta)$, the skeleton with sites $j=0,\dots,n-1$, involution $j\mapsto n-1-j$ and $r_j=e^{(j-(n-1)/2)\beta}$, so that the values $\log r_j$ form a symmetric arithmetic progression with step $\beta$. It has $f=1$ if $n$ is odd (the central site) and $f=0$ if $n$ is even, and
\[
\Phi\bigl(\cP_n(\beta)\bigr)=\sum_{j=0}^{n-1}x^{(j-\frac{n-1}2)\beta}=y^{-\frac{n-1}2}\,[n]_y,\qquad y=x^\beta,\quad [n]_y=1+y+\dots+y^{n-1},
\]
which is a symmetric $q$-integer. The degree is arbitrary and plays no role in the reduced monoid. For instance, $\cP_2(\beta)=\cO(\beta/2)$, $\cP_3(\beta)=\cT(\beta)$, and $\cP_1(\beta)$ is a unit.

These skeletons appear naturally in the problem from which the theory arose. Consider the biased random walk on $\{0,\dots,a\}$, absorbed at the endpoints, which jumps to the right with probability $p$ and to the left with probability $q=1-p$. Its $a-1$ interior states, with the reflection $z\mapsto a-z$ as involution and $\log r_z=(z-\tfrac a2)\log(p/q)$, form a progression skeleton: it is $\cP_{a-1}(\beta)$ with $\beta=|\log(p/q)|$ (Remark~\ref{rem:ruin}). We will see in Section~\ref{subsec:ruin} that the arithmetic of this family is very rigid.
\end{example}

\section{Cancellation, characters and Grothendieck completions}\label{sec:cancel}

In Section~\ref{sec:setting} we reduced the study of skeletons to that of the pairs $(P,f)$ in $\Sy^\sharp$, with the product $(P,f)(Q,g)=(PQ,fg)$. Before turning to factorization, which is the central subject of the paper, it is worth understanding this product well, and the first natural question is whether one can simplify: if $A\times C\cong B\times C$, does it follow that $A\cong B$? In the integers, or in a polynomial ring, the answer is yes, and almost every factorization technique relies on it. Here the answer is no, and the reason is simple: the component $f$ multiplies like a natural number, and multiplying by $0$ erases the information.

This section studies that failure of cancellation, and its consequences, in three steps: first we describe exactly when cancellation fails (Theorem~\ref{thm:nocancel}). Then we introduce a family of characters, real numbers that multiply under the product and add under the disjoint union, and that together with the degree and $f$ determine the skeleton. Finally we construct the Grothendieck group and ring: the multiplicative group, which allows us to divide skeletons formally, and the ring, which allows us to subtract them. In both, the failure of cancellation leaves a precise trace: the group forgets $f$, and the ring has zero divisors.

\subsection{The failure of cancellation}

The example from the introduction, $\cO(t)\times\cF_2\cong\cO(t)\times\cO(0)$, is the typical case. The following theorem says that it is also the only mechanism: cancellation fails only when the factor to be cancelled has no fixed points, and then the only thing lost is $f$.

\begin{theorem}[The monoid is not cancellative]\label{thm:nocancel}
For every $t\ge0$ we have
\[
\cO(t)\times\cF_2\ \cong\ \cO(t)\times\cO(0),\qquad\text{but}\qquad\cF_2\not\cong\cO(0).
\]
More generally, $(P,f)(Q,0)=(P,f')(Q,0)$ for all $(P,f),(P,f')\in\Sy^\sharp$: multiplying by a skeleton without fixed points forgets the number of fixed points. Conversely, if $(P,f)(R,h)=(P',f')(R,h)$, then $P=P'$, and $f=f'$ unless $h=0$.
\end{theorem}

\begin{proof}
Both sides of the isomorphism equal $(2x^t+2x^{-t},0)$, while $\cF_2=(2,2)\ne(2,0)=\cO(0)$. In general, $f\cdot0=f'\cdot0$, and the first component does not depend on $f$.

For the converse we have to cancel in each component. In the second one, $fh=f'h$ with $h\ne0$ gives $f=f'$. In the first one we use that $\ZR$ is an integral domain. This is because the group $(\mathbb R,+)$ is totally ordered: if $P,R\ne0$ and $s$, $t$ are their largest exponents, the term of largest exponent in $PR$ is $c_s(P)\,c_t(R)\,x^{s+t}\ne0$, so that $PR\ne0$. Thus $PR=P'R$ with $R\ne0$ gives $(P-P')R=0$, and therefore $P=P'$.
\end{proof}

The proof shows that the failure of cancellation lives only in the component $f$, while the component $\Phi$ behaves as in a polynomial ring. The following remark sums this up.

\begin{remark}
The skeletons with $f=0$ form an ideal for the product: if one factor has no fixed points, neither has the product. In this ideal, two fixed points and a balanced orbit ($t=0$) can no longer be told apart; outside it, no information is lost: if the common factor has fixed points, equality of the products forces equality of both components. Cancellation holds in $\Sy$ and fails only through the component $f$. In Section~\ref{sec:burnside} we give a dynamical reading of this fact (Remark~\ref{rem:nocancel-dyn}).
\end{remark}

A similar phenomenon has long been known in graph theory, and the comparison helps to place the result.

\begin{remark}[Comparison with the direct product of graphs]\label{rem:lovasz}
For the direct product of graphs, $A\times C\cong B\times C$ implies $A\cong B$ whenever the cancelled factor $C$ is not bipartite \cite{Lovasz1971} (see also \cite{ImrichKlavzarRall}). If $C$ is bipartite, cancellation can fail; under the hypotheses of \cite{Hammack2008}, the failures are characterized through involutions that interchange the two parts.

Here the obstruction is more elementary: by Theorem~\ref{thm:nocancel}, multiplying by a skeleton without fixed points forgets the number of fixed points. Such a skeleton, for instance an orbit $\cO(t)$, plays a role similar to that of a bipartite factor, with $\cO(t)$ in the place of $K_2$, and the isomorphism $\cO(t)\times\cF_2\cong\cO(t)\times\cO(0)$ is reminiscent of a bipartite double cover. The analogy is limited to this mechanism; the two settings and their proofs are different.
\end{remark}

\subsection{Characters}

To decide whether two skeletons are isomorphic, or whether one can be a factor of another, it is useful to have numerical invariants that respect the operations. The number of sites $m$ is one of them: it adds under disjoint union and multiplies under the product. But $m$ is very poor, because it does not see the ratios. The idea is to replace the variable $x$ in $\Phi$ by a positive real number: since $x^sx^t=x^{s+t}$, the result is still multiplicative, and now the ratios play a role. Letting the value of $x$ vary, we obtain a continuous family of invariants; taking $x=e^\lambda$, each site $z$ contributes $r_z^{\,\lambda}$. By a \emph{character} of a commutative semiring with identity we mean a semiring homomorphism to $(\mathbb R,+,\cdot)$, that is, an additive and multiplicative map sending zero to $0$ and the identity to $1$.

\begin{proposition}[A continuous family of characters]\label{prop:chi}
For $\lambda\in\mathbb R$, let
\[
\chi_\lambda(\cE)=\sum_{z\in Z}r_z^{\,\lambda}=f+\sum_O2\cosh(\lambda t_O).
\]
Then $\chi_\lambda$ is additive for the disjoint union and multiplicative for the product. Moreover, $K$, $f$ and the function $\lambda\mapsto\chi_\lambda(\cE)$ determine the class of $\cE$.
\end{proposition}

\begin{proof}
The expression $f+\sum_O2\cosh(\lambda t_O)$ is obtained, as for $\Phi$, by grouping the sites: a fixed point contributes $1$, and an orbit of ratio $t_O$ contributes $e^{\lambda t_O}+e^{-\lambda t_O}=2\cosh(\lambda t_O)$. By construction, $\chi_\lambda$ is $\Phi$ followed by the ring homomorphism $\ZR\to\mathbb R$, $x^t\mapsto e^{\lambda t}$, and this is why it inherits the additivity and multiplicativity of $\Phi$ (Proposition~\ref{prop:mult}). The functions $\lambda\mapsto e^{\lambda t}$, for distinct $t$, are linearly independent, by Dedekind's lemma, since they are distinct characters of the group $(\mathbb R,+)$ with values in the multiplicative group $\mathbb R^\times=\mathbb R\setminus\{0\}$; hence the function $\chi_\bullet$ determines the coefficients of $\Phi$, and Theorem~\ref{thm:semiring} applies.
\end{proof}

Extended additively to families, $\chi_\lambda\bigl((a_K)_K\bigr)=\sum_K\chi_\lambda(a_K)$, each $\chi_\lambda$ is a character of the graded semiring $\mathfrak S$ of Definition~\ref{def:graded}: it vanishes on the empty skeletons, takes the value $1$ on the identity, and is multiplicative because the product of two families is a disjoint union of products of homogeneous classes. In particular, it is a character of $\Sy^\sharp_0$.

In particular, the family $\chi_\lambda$ contains the invariants we already knew, and it has a natural interpretation.

\begin{remark}
We have $\chi_0=m$. Since $\cosh\ge1$, also $\chi_\lambda\ge m$, with equality for some $\lambda\ne0$ if and only if all the ratios vanish. The function $\log\chi_\lambda$ is additive under the product. The multiplicative invariants are thus $K$, $f$, $m$ and the family $\chi_\lambda$. The number of orbits $k$ is \emph{not} multiplicative: $k_{12}=(m_1m_2-f_1f_2)/2$ (Proposition~\ref{prop:orbits}).
\end{remark}

Since $\log\chi_\lambda$ is additive, so are its derivatives at $\lambda=0$. The first one vanishes by symmetry; the second gives an invariant with a very clear meaning, the mean of the squares of the $\log r_z$, which measures how far the weights are from balance.

We call \emph{spread} of $\cE$ the number
\[
s^2(\cE)=\frac1m\sum_{z\in Z}(\log r_z)^2=\frac2m\sum_Ot_O^2 .
\]

\begin{proposition}[Spread]\label{prop:spread}
We have $s^2(\cE)=\dfrac{d^2}{d\lambda^2}\log\chi_\lambda\Big|_{\lambda=0}$, and the spread is additive: $s^2(\cE_1\times\cE_2)=s^2(\cE_1)+s^2(\cE_2)$.
\end{proposition}

\begin{proof}
We have $\chi_0=m$, $\chi'_0=\sum_z\log r_z$ and $\chi''_0=\sum_z(\log r_z)^2$. By symmetry, $\chi'_0=0$, so the second derivative of $\log\chi_\lambda$ at $0$, which is $\chi''_0/\chi_0-(\chi'_0/\chi_0)^2$, equals $\chi''_0/\chi_0=s^2$. Since $\log\chi_\lambda$ is additive under the product, so is its second derivative.
\end{proof}

The spread is easy to compute, and it gives a first quick test in factorization problems.

\begin{example}
For the progression skeleton of Example~\ref{ex:ruin}, $s^2(\cP_n(\beta))=\beta^2(n^2-1)/12$. The spread gives a quick necessary condition for factorization: if $\cE=A\times B$, then $s^2(A)+s^2(B)=s^2(\cE)$ and $m_Am_B=m$. For instance, it is consistent with $\cP_{rs}(\beta)\cong\cP_r(\beta)\times\cP_s(r\beta)$ (Theorem~\ref{thm:ruin}):
\[
\frac{\beta^2(r^2-1)}{12}+\frac{r^2\beta^2(s^2-1)}{12}=\frac{\beta^2(r^2s^2-1)}{12}.
\]
\end{example}

\subsection{The multiplicative group}

In the positive integers, passing to the positive rationals allows us to divide. The general construction that does the same for any commutative monoid is its \emph{Grothendieck group} $K_0$: its elements are formal quotients $a/b$, and two quotients $a/b$ and $c/d$ are identified when $a\,d\,e=c\,b\,e$ for some $e$ in the monoid. (We keep the name $K_0$ for this multiplicative group; the additive completion appears below as the Grothendieck ring.) If the monoid is cancellative, the factor $e$ is superfluous and $K_0$ is simply its group of fractions, like $\mathbb Q_{>0}$ for $\mathbb N$. If it is not, the map from the monoid to $K_0$ identifies distinct elements, and Theorem~\ref{thm:nocancel} tells us exactly which ones.

\begin{theorem}[$K_0$ and what it forgets]\label{thm:K0}
For the multiplicative monoid,
\[
K_0\ \cong\ \mathbb R_{>0}\times\Gr(\Sy),
\]
where $\Gr(\Sy)=\{P/Q:\ P,Q\in\Sy\}$ inside the field of fractions of $\ZR$. Two skeletons have the same class in $K_0$ if and only if they have the same degree and the same $\Phi$. The canonical map to $K_0$ is therefore not injective, and what it loses is exactly the number of fixed points.
\end{theorem}

\begin{proof}
Two elements $a,b$ have the same class in $K_0$ if and only if $ae=be$ for some $e$. If $(P,f)(R,h)=(P',f')(R,h)$, the first component gives $PR=P'R$ with $R\ne0$, and $P=P'$ because $\ZR$ is a domain. Conversely, by Theorem~\ref{thm:nocancel}, applied with the factor $\cO(0)=(2,0)$, we have $(P,f)(2,0)=(P,f')(2,0)$, and in a group the common factor cancels: $[(P,f)]=[(P,f')]$ in $K_0$. So $K_0$ only sees the component $P$. The projection $(P,f)\mapsto P$ onto $\Sy$ is surjective (it suffices to take $f\in\{0,1\}$ with $f\equiv c_0(P)$ modulo $2$). Hence $K_0(\Sy^\sharp)=K_0(\Sy)$. Since $\Sy$ is cancellative (Theorem~\ref{thm:nocancel}), $K_0(\Sy)$ is its group of fractions, which lies inside the field of fractions of the domain $\ZR$. Finally, by Theorem~\ref{thm:semiring} the multiplicative monoid is the product $\mathbb R_{>0}\times\Sy^\sharp$ (degree and pair $(\Phi,f)$), and the Grothendieck group of a product is the product of the Grothendieck groups. The factor $\mathbb R_{>0}$ is already a group under the product (the inverse of $K$ is $1/K$), so it is its own Grothendieck group, and $K_0\cong\mathbb R_{>0}\times\Gr(\Sy)$.
\end{proof}

The theorem reduces the description of $K_0$ to that of the group $\Gr(\Sy)$: which quotients $P/Q$ of symmetric elements of $\ZR$ are quotients of skeletons? There is an obvious necessary condition. A skeleton takes positive values at every character $x\mapsto e^u$ (it is a sum of exponentials with nonnegative coefficients), and therefore so does a quotient of skeletons. When the exponents are commensurable, this condition is also sufficient, and the reason is a classical theorem of Pólya: if a polynomial $p$ is positive on $[0,\infty)$ and has positive leading coefficient, then $(1+y)^Mp(y)$ has all its coefficients positive for some $M$. In other words, after multiplication by a power of a simple skeleton, every positive element becomes a skeleton.

\begin{theorem}[Commensurable ratios]\label{thm:polya}
Let $\alpha>0$ and let $R=\sum_jc_jx^{j\alpha}$ be a symmetric element of $\ZR$ with exponents in $\alpha\mathbb Z$. Then $R\in\Gr(\Sy)$ if and only if $R(e^{u})>0$ for every real $u$. In that case one can take $R=P/Q$ with $Q=(x^{\alpha/2}+x^{-\alpha/2})^{2N}$ for some $N$.
\end{theorem}

\begin{proof}
\emph{Necessity.} The elements of $\Sy$ are positive at every $x=e^u$, and hence so are their quotients.

\emph{Sufficiency.} Let $y=x^\alpha$ and let $d$ be the largest exponent of $R$ in $y$. Then $p(y)=y^dR$ is a polynomial with $p(y)>0$ for $y>0$. Its constant term is the lowest coefficient of $R$, which is positive because $R>0$ near $y=0$; by symmetry it coincides with the leading coefficient of $p$. Thus $p>0$ on $[0,\infty)$ and also at infinity, which is the hypothesis of Pólya's theorem: the homogenization of $p$ is positive on the closed simplex. By Pólya's theorem in one variable \cite{Meissner1911,Polya1928,PowersReznick2001}, $(1+y)^Mp(y)$ has positive coefficients for some $M$, which we may take even, $M=2N$. Then
\[
(y^{1/2}+y^{-1/2})^{2N}R=y^{-N-d}(1+y)^{2N}p(y)
\]
has nonnegative coefficients and is symmetric, so it lies in $\Sy$, as does the multiplier $Q=(y^{1/2}+y^{-1/2})^{2N}$.
\end{proof}

The following examples show that $\Gr(\Sy)$ contains elements with negative coefficients, that is, quotients of skeletons that are not skeletons, and that the size of the required multiplier can be large.

\begin{example}
We have
\[
x-1+x^{-1}=\frac{x^2+1+x^{-2}}{x+1+x^{-1}}\in\Gr(\Sy).
\]
In the notation of Example~\ref{ex:ruin}, this is $\Phi(\cP_3(2))/\Phi(\cP_3(1))$: a quotient of two progression skeletons which is not a skeleton, because it has a negative coefficient. In general, by Theorem~\ref{thm:polya}, a symmetric element of $\ZR$ with commensurable exponents is a quotient of skeletons exactly when its values at all the characters $x^t\mapsto e^{\lambda t}$ are positive.

On the other hand, $x-3+x^{-1}\notin\Gr(\Sy)$, since it takes the value $-1$ at $x=1$. For $x-c+x^{-1}$ with $0<c<2$, the least $N$ in Theorem~\ref{thm:polya} (with $\alpha=1$) grows as the minimum $2-c$ tends to $0$: it equals $1,3,19,199$ for $c=1,\,1.5,\,1.9,\,1.99$.
\end{example}

When the ratios are not commensurable, the exponents generate a group of rank $d\ge2$, the group ring is a ring of Laurent polynomials in $d$ variables, and a character $x^t\mapsto e^{\phi(t)}$, with $\phi$ additive, is simply a point of the positive orthant $(0,\infty)^d$. For a symmetric Laurent polynomial $R$ to lie in $\Gr(\Sy)$ there are two necessary conditions. The first is the one above: $R$ must be strictly positive on the positive orthant. The second uses initial forms. Given a weight vector $w$, that is, a linear form on the exponents, we call \emph{initial form} of $R$, and denote by $\mathrm{in}_w(R)$, the sum of the terms $c_tx^t$ of $R$ for which $w(t)$ is maximal. Initial forms are multiplicative, so $R=P/Q$ gives $\mathrm{in}_w(P)=\mathrm{in}_w(Q)\,\mathrm{in}_w(R)$; since $\mathrm{in}_w(P)$ and $\mathrm{in}_w(Q)$ are nonzero with nonnegative coefficients, $\mathrm{in}_w(R)$ must be strictly positive on the positive orthant. The second condition does not follow from the first, so positivity alone is not sufficient. For example, with $\alpha,\beta$ incommensurable, the symmetric element $R=(x^\alpha-x^\beta)^2+(x^{-\alpha}-x^{-\beta})^2+1$ takes values at least $1$ at every character, but its initial form for a weight with $w(\alpha)=w(\beta)=1$ is $(x^\alpha-x^\beta)^2$, which vanishes where $\phi(\alpha)=\phi(\beta)$; hence $R\notin\Gr(\Sy)$.

\begin{problem}[Incommensurable ratios]\label{prob:gr}
Describe the group $\Gr(\Sy)$ when the ratios are not commensurable. In particular, if $R$ is a symmetric Laurent polynomial such that $R$ and all its initial forms $\mathrm{in}_w(R)$ are strictly positive on the positive orthant, does $R$ lie in $\Gr(\Sy)$?
\end{problem}

This is a multivariate Pólya-type question \cite{Handelman1988,PowersReznick2001}. As motivation we mention the toric extensions of Pólya's theorem by Baldi, Sinn, Telek and Weigert \cite{BSTW}, which give certificates of the form $(\sum_{a\in A}u^a)^Ng$ for sparse polynomials $g$ under a condition on the Newton polytope; by themselves they do not settle the question.

\subsection{The Grothendieck ring}

The group $K_0$ allows us to divide skeletons. The analogous construction for the sum allows us to subtract them: formal differences $a-b$ are adjoined, just as one passes from $\mathbb N$ to $\mathbb Z$. Applied to the disjoint union of skeletons of the same degree, it gives a ring, the \emph{Grothendieck ring}. Since the additive structure is free (Section~\ref{subsec:graded}), the ring is easy to describe; what is interesting is that the failure of multiplicative cancellation reappears in it in a very concrete form: as zero divisors.

\begin{theorem}[The Grothendieck ring and its zero divisors]\label{thm:ring}
Let $\ZR^{\mathrm{sym}}$ be the ring of symmetric elements of $\ZR$. The additive completion of $\Sy^\sharp$ (that is, of each component of fixed degree) is
\[
\mathcal R=\bigl\{(P,f)\in\ZR^{\mathrm{sym}}\times\mathbb Z:\ f\equiv P(1)\pmod 2\bigr\},
\]
which is a ring with the product $(P,f)(Q,g)=(PQ,fg)$: it is the Grothendieck ring of the semiring $\Sy^\sharp_0$ of classes of degree $1$, the only component of $\mathfrak S$ closed under the product. Since the product takes the degrees $K,K'$ to $KK'$, the Grothendieck ring of the graded semiring $\mathfrak S$ of Definition~\ref{def:graded} is
\[
\bigoplus_{K>0}\mathcal R\ \cong\ \mathcal R\otimes_{\mathbb Z}\mathbb Z[\mathbb R_{>0}],
\]
where $\mathbb Z[\mathbb R_{>0}]$ is the group ring of the multiplicative group of degrees.

Moreover:
\begin{enumerate}[nosep]
\item $\mathcal R$ is the fibre product of $\ZR^{\mathrm{sym}}$ and $\mathbb Z$ over $\mathbb F_2$, via $P\mapsto P(1)\bmod2$ and $f\mapsto f\bmod2$;
\item the canonical map $\Sy^\sharp\to\mathcal R$ is injective;
\item a nonzero element of $\mathcal R$ is a zero divisor if and only if one of its coordinates vanishes.
\end{enumerate}
In particular, for $n\ne0$ the elements $(0,2n)$ are zero divisors: $(0,2n)\cdot(Q,0)=0$ for every $(Q,0)\in\mathcal R$ (that is, with $Q(1)$ even). The failure of cancellation in Theorem~\ref{thm:nocancel} is exactly the existence of these zero divisors: $\cF_2-\cO(0)=(0,2)$.
\end{theorem}

Before the proof, it is worth reading the statement. The condition $f\equiv P(1)$ modulo $2$ is the trace, in the ring, of the parity condition of $\Sy^\sharp$: for a skeleton, $P(1)=m$ and $m\equiv f$ modulo $2$. The inequality $f\le c_0(P)$, on the other hand, disappears, because formal differences do not respect it. Point~(1) says that $\mathcal R$ consists of the pairs that agree modulo $2$, and point~(3) that the only way for a product to vanish is that each factor vanishes in a different component. The element $(0,2)=\cF_2-\cO(0)$ is the difference of the two skeletons of Theorem~\ref{thm:nocancel}, which are different but give the same result when multiplied by $\cO(t)$. Hence $(0,2)\cdot\cO(t)=\cF_2\,\cO(t)-\cO(0)\,\cO(t)=0$, and $(0,2)$ is a zero divisor.

\begin{proof}
Every difference of elements of $\Sy^\sharp$ satisfies $f\equiv P(1)$ modulo $2$, because the non-constant coefficients of a symmetric element come in equal pairs. Conversely, every $(P,f)\in\mathcal R$ is a difference of elements of $\Sy^\sharp$. Write $P=P^+-P^-$, with $P^\pm\in\NR$ symmetric (the positive and negative parts of $P$), and let $a_0=c_0(P^+)$ and $b_0=c_0(P^-)$; then $f\equiv a_0-b_0$ modulo $2$. With $N=|f|$ and $M=2N+b_0+1$, the pairs
\[
X=(P^++2M,\ f+b_0+2N),\qquad Y=(P^-+2M,\ b_0+2N)
\]
lie in $\Sy^\sharp$ (they are symmetric and nonzero, the second component lies between $0$ and the constant term, and parity holds), and $(P,f)=X-Y$. The parity condition is preserved by the product, because $P(1)Q(1)\equiv fg$. Additive cancellation holds in $\Sy^\sharp$, so the map to $\mathcal R$ is injective. The degree only labels the components, and the product multiplies the labels; this gives the graded form.

Point~(1) is a restatement of the definition of $\mathcal R$. For (3): if $P\ne0\ne f$, then $(P,f)(Q,g)=0$ forces $PQ=0$ and $fg=0$, that is, $Q=0$ and $g=0$, because $\ZR$ is a domain. Conversely, $(0,f)\cdot(2,0)=0$ and $(P,0)\cdot(0,2)=0$, and the factors $(2,0)$ and $(0,2)$ lie in $\mathcal R$. The statement about $(0,2n)$ is a direct computation.
\end{proof}

This completes the algebraic description of the semiring: we know what can be cancelled, which numerical invariants determine it and what its completions look like. In Section~\ref{sec:fact} we turn to the central problem, factorization, where the failure of cancellation reappears as the mechanism behind non-unique factorization.

\section{Factorization}\label{sec:fact}

In this section we study factorization, the central subject of the paper. In the previous sections we have seen that the classes of skeletons form a semiring, that each class is described by a pair $(\Phi,f)$, and that the product of skeletons is the product of these pairs. The natural question now is the same one asked in the integers: which skeletons cannot be decomposed, and in how many ways does an arbitrary skeleton decompose into indecomposable pieces? In the integers the answer is the fundamental theorem of arithmetic: every number is a product of primes, in exactly one way. We will see that here the first half survives, but the second fails in several ways, and that the reasons for this failure can be described precisely.

It is convenient to fix the framework first. We work with the multiplicative monoid, and its units are the one-point skeletons (Remark~\ref{rem:units}), so factorizations are considered up to units, that is, in the reduced monoid $\Sy^\sharp$. In particular, the degree plays no role, which is why we omit it from the blocks. A non-unit is \emph{irreducible} if it is not a product of two non-units. Two factorizations are identified if one is obtained from the other by multiplying the factors by units and reordering them. Following \cite{GHK}, the monoid is \emph{atomic} if every non-unit is a product of irreducibles (the irreducibles are also called \emph{atoms}).

The section proceeds from the general to the particular. First we see that every skeleton factors into irreducibles and that the lengths of its factorizations are bounded, because the number of sites $m$ is multiplicative. Then we study how an orbit multiplies an arbitrary skeleton: the \emph{centre lemma}, which is the basic tool for everything that follows. With it we obtain exact irreducibility criteria for skeletons without fixed points, describe what happens in each type $(m,f)$ and, finally, measure how badly uniqueness fails: distinct factorizations, distinct lengths and infinite elasticity, and we see that distinct lengths persist in the submonoid of skeletons with a single fixed point.

\subsection{Atomicity and lengths}

The first step is to make sure that factorizations exist and cannot grow longer without control. The number of sites is enough for this. It is multiplicative, takes positive integer values and equals $1$ exactly on the units (Remark~\ref{rem:units}), so every non-unit has $m\ge2$. This is the role played by the absolute value in the factorization of integers. The \emph{length} of a factorization is its number of irreducible factors. The monoid is \emph{half-factorial} if all factorizations of each element have the same length. For an integer $m\ge1$, we write $\Omega(m)$ for the number of prime factors of $m$ counted with multiplicity; for instance, $\Omega(12)=3$, since $12=2\cdot2\cdot3$.

\begin{proposition}[Atomic, with bounded lengths]\label{prop:atomic}
Every non-unit of $\Sy^\sharp$ is a product of irreducibles, and if $\cE$ has $m$ sites, every factorization of $\cE$ into irreducibles has length at most $\Omega(m)$. In particular, the monoid is atomic and has bounded factorization lengths.
\end{proposition}

\begin{proof}
The number of sites $m$ is multiplicative (Proposition~\ref{prop:mult}), and every non-unit has $m\ge2$. If $\cE$ is not irreducible, it is a product of two factors with fewer sites, and induction on $m$ gives the existence of a factorization into irreducibles. If $\cE=\cE_1\times\dots\times\cE_l$ with all factors irreducible, then $m=m_1\cdots m_l$ with all $m_i\ge2$, and therefore $l\le\Omega(m)$.
\end{proof}

Parity gives a first, very useful restriction on the possible factors: the number of sites and the number of fixed points multiply separately, and each factor has to respect the congruence $m\equiv f$ modulo $2$.

\begin{proposition}[Parity]\label{prop:parity}
If $\cE=\cE_1\times\cE_2$, then $m=m_1m_2$, $f=f_1f_2$ and $m_i\equiv f_i$ modulo $2$.
\end{proposition}

\begin{proof}
The first two identities are Proposition~\ref{prop:mult}. The congruence holds because every skeleton has $m_i=f_i+2k_i$.
\end{proof}

With these two tools, elementary as they are, many irreducibles can already be identified.

\begin{proposition}[Elementary irreducibles]\label{prop:elem}
\begin{enumerate}[nosep]
\item Every skeleton with $m$ prime is irreducible; in particular $\cO(t)$, $\cF_2$ and $\cT(s)$ are.
\item The skeletons $\cF_n$ of degree $1$ form a submonoid isomorphic to $(\mathbb Z_{\ge1},\times)$, and $\cF_n$ is irreducible if and only if $n$ is prime.
\item The skeletons with $m=4$, $f=2$ and $k=1$ are irreducible, whatever the ratio.
\end{enumerate}
\end{proposition}

\begin{proof}
(1) A non-trivial factorization needs $m_1,m_2>1$, which is impossible if $m$ is prime.

(2) For these skeletons $f=m$, which is multiplicative, and $\cF_a\times\cF_b=\cF_{ab}$.

(3) Since $f_1f_2=2$, one factor has $f_i=1$; by parity, this factor has $m_i$ odd, and $m_i$ also divides $4$, so $m_i=1$ and the factor is a unit.
\end{proof}

The last example is instructive: a skeleton with four sites can be irreducible even though $4$ is not prime, because the component $f$ imposes its own restrictions. This is the first sign that the arithmetic of $\Sy^\sharp$ does not reduce to that of the numbers $m$.

\subsection{The centre lemma}

To go beyond parity we need to understand how the ratios combine in a product. The basic case is multiplication by a single orbit. By the splitting of orbits (Remark~\ref{rem:splitting}), multiplying $\cO(t)$ by an orbit of ratio $s$ produces the ratios $t+s$ and $|t-s|$, and $t$ is one of their centres (their half-sum if $t\ge s$, their half-difference if $t<s$). This leads to the following notion, which is the key to the section.

The \emph{centres} of a pair $\{x,y\}$ of nonnegative numbers are $\tfrac{x+y}2$ and $\tfrac{|x-y|}2$. We say that the pair is \emph{centred at} $t$ if $t$ is one of its centres.

\begin{lemma}[Product with an orbit]\label{lem:centre}
\begin{enumerate}[nosep]
\item For $t,s\ge0$, the pair $\{t+s,|t-s|\}$ has centres $\{t,s\}$, and every pair $\{x,y\}$ arises in this way from its centres.
\item If $Q$ has $f_Q$ fixed points and orbits of ratios $s_1,\dots,s_j$, then $\cO(t)\times Q$ has $f=0$ and ratios
\[
\underbrace{t,\dots,t}_{f_Q}\ \cup\ \bigcup_i\{t+s_i,\ |t-s_i|\}.
\]
\item Consequently, a skeleton with $f=0$ is of the form $\cO(t)\times Q$ for some $Q$ if and only if its ratios can be divided into copies of $t$ and pairs centred at $t$. In that case one can take for $Q$ the skeleton with one fixed point for each copy of $t$ and one orbit of ratio $s$ for each pair, where $s$ is the other centre of the pair.
\end{enumerate}
\end{lemma}

\begin{proof}
(1) If $t\ge s$, then $\tfrac{(t+s)+(t-s)}2=t$ and $\tfrac{(t+s)-(t-s)}2=s$; the case $t<s$ is symmetric. Conversely, for $x\ge y$ it suffices to take $t=\tfrac{x+y}2$ and $s=\tfrac{x-y}2$.

(2) Expand $(x^t+x^{-t})\bigl(f_Q+\sum_i(x^{s_i}+x^{-s_i})\bigr)$: each fixed point of $Q$ gives an orbit of ratio $t$, and each orbit of ratio $s_i$ gives, by the splitting, the orbits of ratios $t+s_i$ and $|t-s_i|$. The fixed points multiply: $f=0\cdot f_Q=0$.

(3) This is (2) read backwards, using (1) to recover each $s_i$ as the other centre of its pair.
\end{proof}

One detail deserves to be pointed out, because it will be used several times: a pair $\{t,t\}$ is centred at $t$ (and at $0$). This is why the copies of $t$ in (2) can be grouped two by two into pairs centred at $t$, and only one is left over when their number is odd.

The same computation, carried out for two arbitrary factors, describes all the ratios of a product. As a consequence, the possible ratios of the factors lie in a finite list that can be read off from the skeleton itself, and reducibility can be decided by a finite search.

\begin{proposition}[General product and decidability]\label{prop:general}
\begin{enumerate}[nosep]
\item If $A$ has $f_A$ fixed points and ratios $a_1,\dots,a_i$, and $B$ has $f_B$ fixed points and ratios $b_1,\dots,b_j$, then $A\times B$ has $f_Af_B$ fixed points and ratios
\[
\underbrace{a_1,\dots,a_i}_{f_B\ \text{times each}}\ \cup\ \underbrace{b_1,\dots,b_j}_{f_A\ \text{times each}}\ \cup\ \bigcup_{p,q}\{a_p+b_q,\ |a_p-b_q|\}.
\]
\item Let $R$ be the multiset of ratios of $\cE$. If $\cE\cong A\times B$ with $A,B$ non-units, every ratio of $A$ and of $B$ lies in the finite set of \emph{candidate ratios}
\[
\mathrm{Cand}(R)=R\cup\bigl\{\tfrac{u+v}2,\ \tfrac{|u-v|}2:\ u,v\in R\bigr\},
\]
formed by the ratios of $\cE$ and the half-sums and half-differences of each pair of them.
Consequently, whether $\cE$ is reducible, and what all its factorizations are, is determined by a finite search.
\end{enumerate}
\end{proposition}

\begin{proof}
(1) Expand $\Phi(A)\Phi(B)$, with $\Phi(A)=f_A+\sum_p(x^{a_p}+x^{-a_p})$ and similarly for $B$, using the splitting of orbits: the product of a fixed point and an orbit is an orbit of the same ratio, and the product of two orbits gives the two ratios $a_p+b_q$ and $|a_p-b_q|$. The fixed points are counted by Proposition~\ref{prop:mult}.

(2) Let $a$ be a ratio of $A$. If $f_B>0$, then $a\in R$ by (1). If $f_B=0$, then $B$ has some orbit, of ratio $b$ say, because $m_B\ge2$; by (1), the pair $\{a+b,|a-b|\}$ is contained in $R$, and $a$ is one of its centres (Lemma~\ref{lem:centre}(1)). The same holds for $B$. The multiplicities of the ratios of the factors are bounded by $k$, and the values of $m_A$, $m_B$, $f_A$, $f_B$ by $m$ and $f$; so only finitely many pairs $(A,B)$ need to be tried. Iterating over the factors gives all factorizations.
\end{proof}

The finite search settles any concrete case, but says nothing in general. To obtain explicit criteria we restrict ourselves to skeletons without fixed points, where the centre lemma takes its cleanest form.

\subsection{Exact criteria}

If a skeleton without fixed points has a factor consisting of a single orbit, Lemma~\ref{lem:centre}(3) tells us exactly how to recognize it: its ratios must be divided into copies of $t$ and pairs centred at $t$, for a common $t$. The question is whether it is enough to look for factors of this kind, that is, whether every reducible skeleton without fixed points has an orbit as a factor. Surprisingly, the answer depends only on the arithmetic of the number of orbits $k$, and the boundary between the two behaviours is sharp.

\begin{theorem}[Orbit factors: the exact range]\label{thm:orbitfactor}
We call $k\ge2$ \emph{admissible} if $k\in\{8,16\}$, or $k$ is prime, or $k$ is twice a prime. (For $k=1$ the skeleton is an orbit, which is irreducible.)
\begin{enumerate}[nosep,label=(\alph*)]
\item If $k$ is admissible, a skeleton $\cE$ with $f=0$, $k$ orbits and multiset of ratios $R$ is reducible if and only if it has an orbit factor, $\cE\cong\cO(t)\times Q$. Equivalently:
\begin{enumerate}[nosep,label=(\arabic*)]
\item if $k$ is even, $\cE$ is reducible if and only if $R$ can be divided into $k/2$ pairs with a common centre;
\item if $k$ is odd, $\cE$ is reducible if and only if $R$ contains an element $t$ such that $R\setminus\{t\}$ can be divided into pairs centred at $t$.
\end{enumerate}
\item If $k$ is not admissible, there are reducible skeletons with $f=0$ and $k$ orbits that have no orbit factor.
\end{enumerate}
In particular, the criterion holds for every $2\le k\le8$, and the first value at which it fails is $k=9$ (Example~\ref{ex:k9}).
\end{theorem}

The idea of the proof of (a) is the following: factors with two or four sites and no fixed points are orbits or products of two orbits, so if a factor has this size, there is already an orbit factor. If all factors are larger, the number of orbits must have a specific arithmetic form, and this form is precisely what admissibility excludes. For (b) we construct explicit examples with ``generic'' ratios, where there are no accidental coincidences that would allow centred pairs to be formed.

\begin{proof}
(a) Let $\cE\cong A\times B$ with $A,B$ non-units. Since $f_Af_B=0$, we may assume $f_A=0$, and then $m_A$ is even.

By Theorem~\ref{thm:nocancel}, we may replace $B$ by any $(\Phi(B),f')\in\Sy^\sharp$ without changing $A\times B$. If $m_B$ is even, we take $f'=0$, which lies in $\Sy^\sharp$ because $c_0(\Phi(B))\equiv\Phi(B)(1)=m_B$ modulo $2$.

A skeleton with $f=0$ and $m=2$ is an orbit. One with $f=0$ and $m=4$ has two orbits, of ratios $a,b$ say, and it is $\cO(\tfrac{a+b}2)\times\cO(\tfrac{|a-b|}2)$ by Lemma~\ref{lem:centre}(1). So $\cE$ has an orbit factor if $m_A\in\{2,4\}$ or $m_B\in\{2,4\}$.

Otherwise, $m_A\ge6$ is even, and there are two possibilities:
\begin{itemize}[nosep]
\item $m_B\ge3$ is odd. Then $k=m/2=(m_A/2)\,m_B$ has an odd divisor $g=m_B\ge3$ with $k/g\ge3$.
\item $m_B\ge6$ is even. Then $k=2\,(m_A/2)(m_B/2)$, and $k/2$ is a product of two integers $\ge3$.
\end{itemize}
The first is impossible if $k$ is a power of $2$, a prime or twice a prime. The second is impossible if $k$ is odd, twice a prime, or a power of $2$ not exceeding $16$, because then $k/2\le8<9$. So neither occurs for admissible $k$, and reducible skeletons have an orbit factor. Conversely, $\cO(t)$ has $m=2$ and $Q$ has $m=k\ge2$, so neither is a unit.

For the reformulation: $\cE\cong\cO(t)\times Q$ means, by Lemma~\ref{lem:centre}(3), that $R$ consists of $f_Q$ copies of $t$ and of pairs centred at $t$, with $f_Q\equiv m_Q=k$ modulo $2$. As observed after Lemma~\ref{lem:centre}, the copies of $t$ can be grouped two by two into centred pairs, and one is left over exactly when $k$ is odd. Conversely, a division as in (1) gives a $Q$ with $f_Q=0$ whose ratios are the other centres of the pairs, and a division as in (2) gives a $Q$ with $f_Q=1$.

(b) If $k$ is not admissible, then either $k$ has an odd prime divisor $g$ with $e=k/g\ge3$, or $k=2^h$ with $h\ge5$. Let $a_1,a_2,\dots$ be positive real numbers linearly independent over $\mathbb Q$.

\emph{First case.} Let $A$ be the skeleton with $f=0$ and ratios $a_1,\dots,a_e$, and let $\cE=\cF_g\times A$. It has $f=0$, $k$ orbits, and its multiset of ratios consists of each $a_i$ repeated $g$ times.

\emph{Second case.} Let $A$ and $B$ be the skeletons with $f=0$ and ratios $a_1,\dots,a_4$ and $a_5,\dots,a_{4+e'}$, with $e'=2^{h-3}$, and let $\cE=A\times B$. It has $f=0$, $k=2\cdot4\cdot e'$ orbits, and pairwise distinct ratios $a_i+a_j$ and $|a_i-a_j|$ with $i\le4<j$.

In both cases $\cE$ is reducible. Suppose that $\cE\cong\cO(t)\times Q$; then $R$ consists of copies of $t$ and of pairs centred at $t$, and a pair $\{x,y\}$ has the centres $\tfrac{x+y}2$ and $\tfrac{|x-y|}2$. Every ratio is a rational combination of the $a_i$, and therefore so is $t$.

If $t=0$, each pair consists of two equal elements, and $0\notin R$. This forces all multiplicities to be even, which fails in both cases: in the first, each ratio appears $g$ times, with $g$ odd, and in the second the ratios are distinct.

So $t\ne0$. Since $t$ is a ratio or a centre of a pair, it involves at most two of the $a_i$ in the first case, and at most two indices $\le4$ and two indices $>4$ in the second.

In the first case, since $e\ge3$, some $a_l$ does not appear in $t$. Then $a_l\ne t$, and a pair containing $a_l$ has centres $a_l$ and $0$, or centres in which $a_l$ appears with coefficient $\pm\frac12$; none of them is $t$, which is a contradiction.

In the second case, we choose $p\le4<q$ that do not appear in $t$; this is possible because $4\ge3$ and $e'\ge4$. Then $x=a_p+a_q\ne t$, so $x$ lies in a pair $\{x,y\}$ centred at $t$. Every element of $R$ has exactly two nonzero coefficients, equal to $\pm1$, one at an index $\le4$ and one at an index $>4$. A centre $\tfrac{x+y}2$ or $\tfrac{|x-y|}2$ of the pair has zero coefficients at $a_p$ and $a_q$ only if $y=\pm(a_p+a_q)$, that is, $y=x$; and this is excluded because the ratios are distinct. This contradicts the choice of $p$ and $q$.
\end{proof}

For few orbits, the theorem translates into very concrete rules on the ratios, which allow irreducibility to be decided at a glance.

\begin{corollary}[Few orbits]\label{cor:small}
Let $\cE$ be a skeleton without fixed points with $k$ orbits; we identify it with its multiset of ratios.
\begin{enumerate}[nosep]
\item If $k=2$, $\cE$ is always reducible: $\{a,b\}=\cO(\tfrac{a+b}2)\times\cO(\tfrac{|a-b|}2)$.
\item If $k=3$, $\cE=\{a,b,c\}$ is reducible if and only if one of the three ratios is the half-sum or the half-difference of the other two. In particular, $\{1,2,6\}$, $\{1,2,7\}$, $\{1,2,9\}$ and $\{2,3,11\}$ are irreducible.
\item If $k=4$, $\cE$ is reducible if and only if its four ratios split into two pairs centred at the same $t$.
\end{enumerate}
\end{corollary}

\begin{proof}
(1) This is Lemma~\ref{lem:centre}(1). (2) Since $3$ is prime, Theorem~\ref{thm:orbitfactor}(a)(2) applies: we need a $t\in R$ such that the other two ratios form a pair centred at $t$, that is, $t$ is their half-sum or their half-difference. In the four examples one checks directly that none of the three ratios is. (3) Since $4$ is twice a prime, Theorem~\ref{thm:orbitfactor}(a)(1) applies.
\end{proof}

The examples in (2) show that parity is necessary but not sufficient: they have $m=6$ and $f=0$, which parity would allow to decompose as a factor with $m=2$, $f=0$ times one with $m=3$, $f=1$, and yet they are irreducible.

\begin{example}
$\{1,3,5,7\}$ is reducible: $\{1,7\}$ and $\{3,5\}$ share the centre $4$. It is $\cO(4)\times Q$, where $Q$ is the skeleton with $f=0$ and ratios $\{1,3\}$, the other centres of the two pairs. On the other hand, $\{1,2,6,20\}$ is irreducible.
\end{example}

The theorem says that, for admissible $k$, it is enough to look for orbit factors, but not that every factorization has an orbit factor. The following remark clarifies this point and records the computational checks we have carried out.

\begin{remark}
For $k=6$ there are factorizations in which no factor is an orbit, for instance $\cT(s)\times Q$ with $Q$ having $m=4$ and $f=0$. But, by the proof above, they always also admit an orbit factor, because $Q$ is itself a product of two orbits.

Theorem~\ref{thm:orbitfactor} has been checked against an exhaustive search for $k=6$ (the $924$ multisets of six ratios in $\{0,\dots,6\}$, with factors ranging over all skeletons with half-integer ratios), and against $3600$ random products for $k=8$ and $2700$ for $k=16$, with no discrepancy.
\end{remark}

The first non-admissible value is $k=9=3\cdot3$, and the example that shows it is the simplest possible: three copies of an irreducible with three orbits.

\begin{example}[The first failure: $k=9$]\label{ex:k9}
Let $C$ be the irreducible skeleton with $f=0$ and ratios $\{1,2,6\}$ (Corollary~\ref{cor:small}). Then $\cF_3\times C$ has $f=0$, $k=9$ and ratios $\{1,1,1,2,2,2,6,6,6\}$. It is reducible, but it has no orbit factor, and the criterion of Theorem~\ref{thm:orbitfactor}(a) fails at $k=9$, the smallest non-admissible value.

Indeed, $\Phi(\cF_3\times C)=3\,(x^{\pm1}+x^{\pm2}+x^{\pm6})$. If $\cF_3\times C\cong\cO(t)\times Q$, then $\Phi=(x^t+x^{-t})\,\Phi(Q)$, that is, $c_y(\Phi)=c_{y-t}(\Phi(Q))+c_{y+t}(\Phi(Q))$ for all $y$. If $t=0$, $c_1(\Phi)=2c_1(\Phi(Q))$ would be even, but it equals $3$. If $t\ne0$ and $y\in\{1,2,6\}$, one of the two summands is nonzero, and then $x^{y-2t}$ or $x^{y+2t}$ occurs in $\Phi$. So $|2t|$ lies in $\{|y-z|:z=\pm1,\pm2,\pm6\}\setminus\{0\}$, which is $\{1,2,3,5,7\}$, $\{1,3,4,8\}$ and $\{4,5,7,8,12\}$ for $y=1,2,6$, and these three sets have no common element.

The phenomenon does not depend on $\cF_3$: the product $C\times\cT(1)$ has ratios $\{0,1,1,2,2,3,5,6,7\}$ and has no orbit factor either, and the same happens for $C\times\cT(s)$ for every integer $1\le s\le29$.
\end{example}

A criterion whose validity depends on the arithmetic of a number of components is not peculiar to skeletons; a phenomenon of the same kind occurs in graph theory.

\begin{remark}[Comparison with disconnected graphs]
The shape of Theorem~\ref{thm:orbitfactor}, a criterion whose validity depends on the arithmetic of the number of orbits, resembles, with different exceptional values, van de Woestijne's analysis of the factorizations of disconnected graphs for the Cartesian, strong and direct products \cite{vdWoestijne2012}. There, factorization is unique when the number of components is prime or equal to $1$, $4$, $8$ or $9$ (under additional conditions), and the non-unique factorizations with $6$ and $10$ components are described explicitly. In both cases the question reduces to factoring in a semiring of polynomials with nonnegative coefficients, and semirings of graphs embed in monoid algebras much as $\Phi$ embeds $\Sy$ in $\ZR$ \cite{ImrichKlepSmertnig}.
\end{remark}

\subsection{Types and genericity}

The criteria above look at the actual ratios. One may also ask what can be said by looking only at the size and the number of fixed points. We call the pair $(m,f)$ the \emph{type} of a skeleton. The type fixes the number of orbits, $k=(m-f)/2$, but not the ratios. Once the type is fixed, the ratios of the $k$ orbits, in any order, form a vector $(t_1,\dots,t_k)$ in the \emph{space of ratios} $[0,\infty)^k$; by Theorem~\ref{thm:class}, this vector, up to the order of the coordinates, determines the skeleton. By Proposition~\ref{prop:parity}, for a skeleton of type $(m,f)$ to factor, the type has to decompose in a way compatible with parity. We say that a type is \emph{decomposable} if there are $m=m_1m_2$ with $m_1,m_2\ge2$ and $f=f_1f_2$ with $0\le f_i\le m_i$ and $f_i\equiv m_i$ modulo $2$.

\begin{proposition}[Arithmetic irreducibility]\label{prop:types}
\begin{enumerate}[nosep]
\item Every skeleton of a non-decomposable type $(m,f)$ is irreducible. If the type is decomposable, some skeleton of that type is reducible.
\item Every skeleton with $k=1$ is irreducible.
\item The skeletons with $k=0$ are the $\cF_n$, irreducible if and only if $n$ is prime. The skeletons of type $(4,0)$ are all reducible.
\end{enumerate}
\end{proposition}

\begin{proof}
(1) The first statement is Proposition~\ref{prop:parity}. For the second it suffices to take the product of any two skeletons of the factor types.

(2) By Proposition~\ref{prop:general}(1), a product $A\times B$ has $f_Bk_A+f_Ak_B+2k_Ak_B$ orbits. If this equals $1$, then $k_Ak_B=0$, say $k_B=0$, and $f_Bk_A=1$, so $B$ is the one-point skeleton, a unit.

(3) The first statement is Proposition~\ref{prop:elem}(2) and the second is Corollary~\ref{cor:small}(1).
\end{proof}

So decomposable types do contain reducible skeletons, but are there many or few? The following theorem answers that they are exceptional: reducibility requires linear relations with integer coefficients among the ratios, and ``random'' ratios do not satisfy them.

\begin{theorem}[Generic irreducibility]\label{thm:generic}
Let $\cE$ have $k\ge1$ orbits whose ratios are linearly independent over $\mathbb Q$. Then $\cE$ is irreducible, unless $(m,f)=(4,0)$. Moreover, in every type other than $(4,0)$ and $(n,n)$, the reducible skeletons lie in a finite union of proper rational linear subspaces of the space of ratios, and the irreducible ones are generic.
\end{theorem}

\begin{proof}
Let $\cE\cong A\times B$ with $A,B$ non-units. By linear independence, the ratios $a_1,\dots,a_k$ are positive and distinct, so the nonzero exponents of $\Phi(\cE)=f+\sum_i(x^{a_i}+x^{-a_i})$ are the $\pm a_i$, all with coefficient $1$. Since $\Phi(A)$ and $\Phi(B)$ have nonnegative coefficients, if $x^u$ occurs in $\Phi(A)$ and $x^v$ in $\Phi(B)$, the coefficient of $x^{u+v}$ in $\Phi(\cE)$ is at least the product of the two coefficients.
\begin{enumerate}[nosep,label=(\roman*)]
\item $\Phi(A)$ has an exponent $u>0$. Otherwise, by symmetry, $\Phi(A)=m_A\ge2$, and all the coefficients of $\Phi(\cE)=m_A\Phi(B)$ would be multiples of $m_A$, but $c_{a_1}(\Phi(\cE))=1$. Likewise, $\Phi(B)$ has an exponent $v>0$.
\item $u+v>0$ is an exponent of $\Phi(\cE)$, so $u+v=a_i$ for some $i$. Since $-v$ is an exponent of $\Phi(B)$, $d=u-v$ is also an exponent of $\Phi(\cE)$, so $d\in\{0,\pm a_j\}$. Thus $2u=a_i+d$ and $2v=a_i-d$.
\item $0$ is not an exponent of $\Phi(A)$. If it were, $v$ would be an exponent of $\Phi(\cE)$, $v=a_l$, and $2a_l-a_i+d=0$. The coefficient of $a_l$ in this relation is $2$ minus at most $1$ from $-a_i$ and at most $1$ from $d$; for it to vanish we need $i=l$ and $d=-a_l$, and then $u=(a_i+d)/2=0$, which is absurd. In particular $f_A\le c_0(\Phi(A))=0$; the same holds for $B$.
\item $u$ is the only positive exponent of $\Phi(A)$. If $u'\ne u$ is another one, then $u'+v=a_{i'}$ with $i'\ne i$, $u'-v=d'\in\{0,\pm a_{j'}\}$, and $a_i-d=a_{i'}-d'$. The coefficient of $a_i$ in this relation vanishes only if $d=a_i$, and then $v=0$, or if $d'=-a_i$, and then $d=-a_{i'}$ and $u'=-u<0$. Both are absurd.
\item The coefficient of $x^u$ in $\Phi(A)$ is $1$, because that of $x^{a_i}=x^{u+v}$ in $\Phi(\cE)$ is $1$.
\end{enumerate}
Hence $\Phi(A)=x^u+x^{-u}$, and $A$ is an orbit; likewise, $B$ is an orbit. Then $f=0$ and $m=4$: the type is $(4,0)$, which is excluded.

For the last statement: by Proposition~\ref{prop:general}(1), the vector of ratios of a product of given types is a piecewise linear function, with integer coefficients, of the $k_A+k_B$ ratios of the factors, and $k_A+k_B<k$. Indeed, if $k_A=0$ or $k_B=0$, say $k_A=0$, then $A=\cF_{f_A}$ with $f_A\ge2$ and $k=f_Ak_B>k_B$; otherwise, if $f_A=f_B=0$, then $k=2k_Ak_B>k_A+k_B$ unless $k_A=k_B=1$, and if $f_B\ge1$, then $k\ge k_A+2k_Ak_B>k_A+k_B$. So, up to the finitely many orderings of the coordinates, the reducible skeletons lie in a finite union of proper rational subspaces.
\end{proof}

Table~\ref{tab:types} summarizes the situation for small types.

\begin{table}[ht]
\centering
\small
\begin{tabular}{@{}p{3.3cm}p{6.6cm}p{3.6cm}@{}}
\toprule
Class of types & Types with $m\le12$ & Status\\
\midrule
$k=0$ & $\cF_n$ & irreducible if and only if $n$ is prime\\
$(4,0)$ & $(4,0)$ & always reducible\\
non-decomposable, $k\ge1$ & $m$ prime; $(4,2)$, $(6,4)$, $(8,2)$, $(8,6)$, $(9,5)$, $(9,7)$, & always irreducible\\
 & $(10,4)$, $(10,8)$, $(12,10)$ & \\
decomposable, $k\ge1$, $\ne(4,0)$ & $(6,0)$, $(6,2)$, $(8,0)$, $(8,4)$, $(9,1)$, $(9,3)$, $(10,0)$, $(10,2)$, & generically irreducible;\\
 & $(10,6)$, $(12,0)$, $(12,2)$, $(12,4)$, $(12,6)$, $(12,8)$ & reducible on a null set\\
\bottomrule
\end{tabular}
\caption{Irreducibility by types for $m\le12$ (Proposition~\ref{prop:types} and Theorem~\ref{thm:generic}). Random samples of each decomposable type were checked by exhaustive factorization of $\Phi$ over $\mathbb Z$.}
\label{tab:types}
\end{table}

\subsection{Non-unique factorization}

We now know that factorizations exist and have bounded length. The second half of the question remains: are they unique? They are not, and what is interesting is that uniqueness fails for two reasons of a different nature. One is arithmetic and comes from the ratios: the same multiset can be divided into centred pairs in two ways. The other comes from the fixed points, and it is the failure of cancellation of Theorem~\ref{thm:nocancel}.

\begin{theorem}[Factorization is not unique]\label{thm:nonunique}
\begin{enumerate}[nosep]
\item $\cO(3)\times\cT(2)\cong\cO(1)\times\cT(4)$: both have $f=0$ and ratios $\{1,3,5\}$, the four factors are irreducible and $\cO(3)\not\cong\cO(1)$.
\item $\cO(t)\times\cF_2\cong\cO(t)\times\cO(0)$, with irreducible factors and $\cF_2\not\cong\cO(0)$.
\end{enumerate}
\end{theorem}

\begin{proof}
(1) By Lemma~\ref{lem:centre}(2), $\cO(3)\times\cT(2)$ has the ratio $3$ (from the fixed point of $\cT(2)$) and the pair $\{5,1\}$, and $\cO(1)\times\cT(4)$ has the ratio $1$ and the pair $\{5,3\}$. Both have $f=0$ and ratios $\{1,3,5\}$, so they are isomorphic by Theorem~\ref{thm:semiring}. The factors have $m=2$ or $m=3$, which are prime, so they are irreducible (Proposition~\ref{prop:elem}(1)), and $\cO(3)\not\cong\cO(1)$ because they have different ratios. Moreover, $\cO(3)\not\cong\cT(4)$, since they have $m=2$ and $m=3$; so the two factorizations do not differ only in the order.

(2) This is Theorem~\ref{thm:nocancel}; the factors have $m=2$.
\end{proof}

The lack of uniqueness could be harmless if all factorizations of each element at least had the same length, that is, if the monoid were half-factorial. This is not the case either, and the smallest possible example already shows it.

\begin{theorem}[The monoid is not half-factorial]\label{thm:halffact}
Let $X$ be the skeleton with $m=4$, $f=2$ and one orbit of ratio $t>0$. Then
\[
X\times\cO(t)\ \cong\ \cO(t/2)\times\cO(t/2)\times\cO(t),
\]
with all factors irreducible. The length set of this element is $\{2,3\}$, so the monoid is not half-factorial. The example has the smallest possible size: two distinct lengths require a factorization of length at least $3$, hence $\Omega(m)\ge3$ by Proposition~\ref{prop:atomic}, hence $m\ge8$.
\end{theorem}

\begin{proof}
Both sides have $f=0$ and $\Phi=x^{2t}+2x^t+2+2x^{-t}+x^{-2t}$, that is, ratios $\{0,t,t,2t\}$. $X$ is irreducible by Proposition~\ref{prop:elem}(3), a special case of Lemma~\ref{lem:f2irr}, and the orbits by Proposition~\ref{prop:elem}(1). Since $m=8$, Proposition~\ref{prop:atomic} bounds every length by $3$, and the element is reducible; so its length set is $\{2,3\}$.
\end{proof}

\begin{remark}
The two phenomena of non-uniqueness in Theorem~\ref{thm:nonunique} have different origins. The first comes from the fact that $\{1,3,5\}$ contains both a half-sum ($3$) and a half-difference ($1=\tfrac{5-3}2$). The second comes from the fixed points. Theorem~\ref{thm:halffact} is the symmetric, real-exponent version of the result of Campanini and Facchini for $\mathbb N_0[x]$, with an explicit element of minimal size ($m=8$).
\end{remark}

\subsection{Elasticity}

Once lengths can differ, the quantitative question is by how much. The usual measure is the \emph{elasticity} of an element, the quotient of its largest and smallest factorization lengths, and that of the monoid is the supremum of these. We will see that the elasticity of elements is unbounded, that is, the elasticity of the monoid is infinite. The cause is, once again, the failure of cancellation. There are irreducibles with $f=2$ that differ from a product of many orbits only in their fixed points; multiplying both by a factor without fixed points makes that difference disappear, and the same element then has a very short factorization and a very long one. The following lemma provides the irreducibles we need.

\begin{lemma}\label{lem:f2irr}
If $f=2$ and $m$ is a power of $2$, then $\cE$ is irreducible.
\end{lemma}

\begin{proof}
A factorization $\cE\cong A\times B$ requires $f_Af_B=2$, so one factor, say $A$, has $f_A=1$. Then $m_A$ is odd by parity and divides $m$, so $m_A=1$ and $A$ is a unit.
\end{proof}

\begin{theorem}[Infinite elasticity]\label{thm:elast}
For $j\ge1$ and $t>0$, let $Y_j$ be the skeleton with
\[
\Phi(Y_j)=\bigl(x^{t/2}+x^{-t/2}\bigr)^{2j},\qquad f(Y_j)=2 .
\]
Then $Y_j$ is irreducible and, for every $s\ge0$,
\[
Y_j\times\cO(s)\ \cong\ \cO(t/2)^{2j}\times\cO(s).
\]
This element has factorizations of lengths $2$ and $2j+1$, and its elasticity is exactly $(2j+1)/2$. In particular, the elasticity of the monoid is infinite.
\end{theorem}

\begin{proof}
$Y_j$ is a skeleton: $c_0=\binom{2j}{j}\ge2$, and it is even because, for a symmetric element, $P(1)\equiv c_0(P)$ modulo $2$ (the non-constant coefficients come in pairs), and here $P(1)=m=4^j$; so $(\Phi,2)\in\Sy^\sharp$. Its size is $m=4^j$, so it is irreducible by Lemma~\ref{lem:f2irr}.

Both sides of the isomorphism have $f=0$ and the same $\Phi$, because $\Phi(\cO(t/2))^{2j}=\Phi(Y_j)$; by Theorem~\ref{thm:semiring} they are isomorphic. The left-hand side is a factorization of length $2$, and the right-hand side one of length $2j+1$, with all factors irreducible.

The size is $m=2^{2j+1}$, so every length is at most $2j+1$ by Proposition~\ref{prop:atomic}, and at least $2$ because the element is reducible. Its elasticity is therefore $(2j+1)/2$, which is unbounded in $j$.
\end{proof}

\begin{remark}
For $j=1$ and $s=t$, this is Theorem~\ref{thm:halffact}, with $Y_1=X$. In the product, the irreducible $Y_j$ and the reducible $\cO(t/2)^{2j}$ give the same result because they differ only in the number of fixed points and $\cO(s)$ has no fixed points. Thus the unbounded elasticity comes from the component $f$, not from the coefficients of $\Phi$.
\end{remark}

The choice $f=2$ in the theorem is not the only one that works. Several skeletons correspond to the polynomial $\Phi(Y_j)$, one for each value of $f$ allowed by the conditions of $\Sy^\sharp$ ($0\le f\le c_0$, with the same parity as $c_0$). The one with the largest $f$ is irreducible too, although for a different and more algebraic reason.

\begin{remark}[The largest value of $f$]
The value $f=2$ is the smallest positive one allowed by parity. The largest is $f=c_0=\binom{2j}{j}$, and the corresponding skeleton is irreducible too.

Indeed, suppose $(\Phi(Y_j),\binom{2j}{j})=(P,f_A)(Q,f_B)$ with non-unit factors. By the formula in the proof of Lemma~\ref{lem:closed},
\[
\tbinom{2j}{j}=f_Af_B\le c_0(P)\,c_0(Q)\le c_0(PQ)=\tbinom{2j}{j},
\]
so $P$ and $Q$ have constant terms and no positive exponent in common. Every exponent of $P$ is then an exponent of $PQ$, hence lies in $t\mathbb Z$, and the same holds for $Q$. With $w=x^t$, both lie in the factorial ring $\mathbb Z[w^{\pm1}]$, where $PQ=w^{-j}(1+w)^{2j}$ and $1+w$ is prime. Therefore $P=w^{-i/2}(1+w)^{i}$ with $i$ even by symmetry, and $2\le i\le 2j-2$ because neither factor is a unit. Thus $P$ and $Q$ both contain the exponent $t$, which is a contradiction.

For $s=t/2$, the element $Y_j\times\cO(t/2)\cong\cO(t/2)^{2j+1}$ does not depend on the chosen value of $f$, and an exhaustive computation shows that its length set is the full interval $\{2,3,\dots,2j+1\}$ for $j\le3$.
\end{remark}

\subsection{\texorpdfstring{The submonoid $f=1$}{The submonoid f=1}}

The two phenomena of non-uniqueness in Theorem~\ref{thm:nonunique} suggest a natural question: if we remove the fixed-point mechanism, is half-factoriality at least recovered? The right place to ask it is among the skeletons with a single fixed point. There the number of fixed points cannot be erased, because all the factors also have one fixed point, and the component $f$ no longer plays a role.

\begin{proposition}\label{prop:f1}
The skeletons with $f=1$ form a divisor-closed submonoid: if $\cE\cong A\times B$ and $f(\cE)=1$, then $f_A=f_B=1$. It is isomorphic to the multiplicative monoid of the $P\in\Sy$ with $P(1)$ odd, and in it the component $f$ plays no role.
\end{proposition}

\begin{proof}
$f_Af_B=1$ forces $f_A=f_B=1$. In this submonoid, $(P,1)\mapsto P$ is injective, and its image consists of the $P$ with $c_0(P)\ge1$ odd, that is, with $P(1)=m$ odd.
\end{proof}

So the mechanism of Theorem~\ref{thm:elast} does not act in the submonoid $f=1$. Even so, the answer to the question is negative: the splitting of the ratios is enough on its own to produce different lengths.

\begin{theorem}[The submonoid $f=1$ is not half-factorial]\label{thm:f1}
Write $\cT(a_1,\dots,a_j)$ for the skeleton with one fixed point and orbits of ratios $a_1,\dots,a_j$. Then
\[
\cT(3)\times\cT(2)\times\cT(2,6)\ \cong\ \cT(4)\times\cT(3)\times\cT(2,4)\ \cong\ \cT(1,3)\times\cT(2,3,4,8),
\]
with all factors irreducible. This element has $m=45$ and length set $\{2,3\}$, so the submonoid $f=1$ is not half-factorial.
\end{theorem}

\begin{proof}
The three products have $f=1$. By Proposition~\ref{prop:general}(1), $\cT(3)\times\cT(2)\cong\cT(1,2,3,5)$ and $\cT(4)\times\cT(3)\cong\cT(1,3,4,7)$; a second application shows that the three products have the same ratios,
\[
\{0,1,1,1,1,2,2,3,3,3,3,4,4,5,5,5,6,7,7,8,9,11\}.
\]
The factors with $m=3$ or $m=5$ are irreducible because $m$ is prime. $\cT(2,3,4,8)$ has $m=9$, so in a factorization $A\times B$ into non-units, $m_Am_B=9$ and $f_Af_B=1$ give $m=3$ and $f=1$ in each factor, that is, a fixed point and an orbit: $A=\cT(a)$ and $B=\cT(b)$, with ratios $\{a,b,a+b,|a-b|\}$; but there are no two elements of $\{2,3,4,8\}$ whose sum and difference both lie in the set. Lengths are at most $\Omega(45)=3$. (An exhaustive search, by Proposition~\ref{prop:general}(2), also confirms that these are all the factorizations.)
\end{proof}

\begin{remark}
Here the different lengths come only from the coefficients of $\Phi$, through the splitting of the exponents, and not from the fixed points. Whether the elasticity of the submonoid $f=1$ is finite remains open.
\end{remark}

In summary, we have seen that the monoid of skeletons is atomic with bounded lengths, that irreducibility is decided by the centre lemma and is the generic situation, and that uniqueness fails both through the fixed points and through the arithmetic of the ratios. In Section~\ref{subsec:ruin} we study a family in which, by contrast, the arithmetic is very rigid: the progression skeletons.

\section{Progression skeletons}\label{subsec:ruin}

In Section~\ref{sec:fact} we described a monoid in which factorization fails in almost every possible way: it is not unique, lengths can differ, and elasticity is unbounded. In this section we study a family in which the opposite happens. The progression skeletons of Example~\ref{ex:ruin}, among them the gambler's ruin skeleton, have a rigid arithmetic: their factorizations can be listed one by one, they all have the same length, and their number is given by a closed formula. The reason is that the factorization of these skeletons is the same as that of number systems.

To see this, recall that $\Phi(\cP_n(\beta))=y^{-\frac{n-1}2}[n]_y$, with $y=x^\beta$ and $[n]_y=1+y+\dots+y^{n-1}$. Writing a number $j\in\{0,\dots,n-1\}$ in a mixed radix, for instance $j=j_1+rj_2$ with $0\le j_1<r$ and $0\le j_2<s$ when $n=rs$, amounts to the polynomial identity
\[
[rs]_y=[r]_y\,[s]_{y^r},
\]
because each power $y^j$ appears exactly once on the right, as $y^{j_1}\cdot y^{rj_2}$. This identity, transferred to skeletons, is point~(1) of the following theorem. That \emph{all} factorizations come from it is classical for polynomials with coefficients in $\{0,1\}$: the correspondence between the factorizations of $1+y+\dots+y^{n-1}$ into such polynomials and the ordered factorizations of $n$ goes back to Krasner and Ranulac \cite{KrasnerRanulac}; see also \cite{deBruijn1956,CarlitzMoser1966}. Point~(3) of the theorem shows that no assumption on the coefficients of the factors is needed, and that the result holds in $\Sy^\sharp$, with its component $f$.

\begin{theorem}[Factorizations of progression skeletons]\label{thm:ruin}
Let $\cP_n(\beta)$ be as in Example~\ref{ex:ruin}.
\begin{enumerate}[nosep]
\item For $r,s\ge1$, $\cP_{rs}(\beta)\cong\cP_r(\beta)\times\cP_s(r\beta)$ in the reduced monoid, via the mixed-radix bijection $j=j_1+rj_2$.
\item $\cP_n(\beta)$ is irreducible if and only if $n$ is prime.
\item The factorizations of $\cP_n(\beta)$ into irreducibles are exactly
\[
\cP_{p_1}(\beta)\times\cP_{p_2}(p_1\beta)\times\dots\times\cP_{p_\ell}(p_1\cdots p_{\ell-1}\beta),
\]
one for each ordering $(p_1,\dots,p_\ell)$, with $\ell=\Omega(n)$, of the prime factors of $n$ counted with multiplicity; distinct orderings give distinct factorizations. In particular, every factorization has length $\Omega(n)$, and the number of factorizations is the multinomial coefficient $\Omega(n)!/\prod_pe_p!$, where $n=\prod_pp^{e_p}$.
\end{enumerate}
\end{theorem}

Before the proof, let us look at an example. For $n=6$ there are two orderings of the prime factors, $(2,3)$ and $(3,2)$, and therefore two factorizations:
\[
\cP_6(\beta)\cong\cP_2(\beta)\times\cP_3(2\beta)\cong\cP_3(\beta)\times\cP_2(3\beta).
\]
In the first, each $j\in\{0,\dots,5\}$ is written $j=j_1+2j_2$, with $j_1\in\{0,1\}$ and $j_2\in\{0,1,2\}$: the factor $\cP_2(\beta)$, with step $\beta$, gives the digit $j_1$, and the factor $\cP_3(2\beta)$, with step $2\beta$, gives the digit $j_2$, which moves $j$ forward two at a time. In the second, $j=j_1+3j_2$, with $j_1\in\{0,1,2\}$ and $j_2\in\{0,1\}$: $\cP_3(\beta)$ gives $j_1$ and $\cP_2(3\beta)$ moves $j$ forward three at a time. It is like writing a number with units and tens, but with a different base for each digit. Both factorizations have length $2=\Omega(6)$.

The proof of (1) is a direct check with the mixed radix. The proof of (3) is the deep part: one has to show that there are no other factorizations. The idea is that all the coefficients of $\Phi(\cP_n(\beta))$ are $0$ or $1$, so any factorization splits the set $\{0,\dots,n-1\}$ into a direct sum of sets of exponents. A classical theorem of de~Bruijn on such direct sums forces one of the factors to be a multiple of a block of consecutive digits $\{0,\dots,r-1\}$, that is, of a progression skeleton; by irreducibility it is that skeleton, and the rest follows by induction.

The classical lemma on which the proof rests is the following. We include its proof to keep the paper self-contained.

\begin{lemma}[de~Bruijn \cite{deBruijn1956}; see also \cite{CarlitzMoser1966}]\label{lem:debruijn}
Let $I,J\subset\mathbb N_0$ with $0\in I\cap J$, such that every $x\in\{0,\dots,n-1\}$ can be written uniquely as $a+b$ with $a\in I$ and $b\in J$, and all these sums are less than $n$. Let $r$ be the smallest positive integer not in $I$ ($r=n$ if $I=\{0,\dots,n-1\}$). Then every element of $J$ is a multiple of $r$, $I$ is a union of complete blocks $\{rq,\dots,rq+r-1\}$, and $r\mid n$.
\end{lemma}

\begin{proof}
If $x\ne0$ were in $I\cap J$, $x=x+0=0+x$ would be two representations; so $I\cap J=\{0\}$. If $r<n$, then $r\in J$: in $r=a+b$ we have $b\ne0$ because $r\notin I$, and $b<r$ would give $b\in I\cap J$.

We prove by induction on $x<n$ that (i) if $x\in J$, then $r\mid x$, and (ii) $x\in I$ if and only if $r\lfloor x/r\rfloor\in I$. Let $x=rq+s$ with $0\le s<r$. If $s=0$ there is nothing to prove, and if $q=0$, then $x\in I$ and $x\notin J$. Let $q,s\ge1$.
\begin{itemize}[nosep]
\item If $rq\in I$: $x\notin J$, since otherwise $(r-s)+x$ and $rq+r$ would be two representations of the same number. Write $x=a+b$; by induction $b=rp$, and if $p\ge1$, (ii) applied to $a=r(q-p)+s$ gives $r(q-p)\in I$, and $rq=r(q-p)+rp=rq+0$ would have two representations. So $b=0$ and $x\in I$.
\item If $rq\notin I$: $rq=a+b$ with $b\ne0$; by induction $b=rp$ with $p\ge1$, $a=r(q-p)\in I$, and by (ii) $y=r(q-p)+s\in I$. Then $x=y+rp$, and $x\notin I$ (otherwise $x+0$ would be another representation) and $x\notin J$ (otherwise $0+x$ would be).
\end{itemize}
The first statement follows from (i) and the second from (ii); the block fits in $\{0,\dots,n-1\}$ because $rq+r$ is a sum (with $r\in J$) or because $r=n$. Finally, write $n-1=a+b$ with $b=rp$; since the block of $a$ is complete, $r\lfloor a/r\rfloor+r-1+b$ is a sum greater than or equal to $n-1$, hence equal to $n-1$, and $n=r(\lfloor a/r\rfloor+p+1)$.
\end{proof}

With de Bruijn's lemma we can now prove the theorem.

\begin{proof}[Proof of Theorem~\ref{thm:ruin}]
(1) Write $j=j_1+rj_2$ with $0\le j_1<r$ and $0\le j_2<s$. Since $\tfrac{r-1}2+r\tfrac{s-1}2=\tfrac{rs-1}2$,
\[
\Bigl(j-\tfrac{rs-1}2\Bigr)\beta=\Bigl(j_1-\tfrac{r-1}2\Bigr)\beta+\Bigl(j_2-\tfrac{s-1}2\Bigr)r\beta,
\]
and $rs-1-j=(r-1-j_1)+r(s-1-j_2)$. The bijection therefore preserves the involutions and the values $\log r_z$; hence it preserves $\Phi$ and $f$, and by Theorem~\ref{thm:semiring} the two skeletons are isomorphic up to the degree, which is invisible in the reduced monoid.

(2) If $n$ is prime, $\cP_n(\beta)$ is irreducible by Proposition~\ref{prop:elem}(1). If $n=1$, it is a unit, and if $n=rs$ with $r,s\ge2$, item (1) writes it as a product of two elements with $m=r$ and $m=s$, which are not units.

(3) Let $\cP_n(\beta)\cong\cG_1\times\dots\times\cG_\ell$ with irreducible factors, and let $P_i=\Phi(\cG_i)$. All the coefficients of $\prod_iP_i=\Phi(\cP_n(\beta))$ are $0$ or $1$, and all the coefficients of the $P_i$ are nonnegative.

\emph{Step 1: the factors also have coefficients $0$ or $1$.} Fix $i$ and let $Q_i=\prod_{l\ne i}P_l$, with largest exponent $q_+$. The coefficient of $x^{a+q_+}$ in $P_iQ_i$ is at least that of $x^a$ in $P_i$; so each $P_i$ has coefficients in $\{0,1\}$. Moreover, differences of exponents of $P_i$, and also those of $Q_i$, are differences of exponents of the product, and therefore lie in $\beta\mathbb Z$.

\emph{Step 2: a direct sum.} For fixed $i$, since $P_iQ_i=\Phi(\cP_n(\beta))$ has $n$ terms with coefficient $1$, the exponents of $P_i$ and of $Q_i$ are $g_i+u\beta$ and $h_i+v\beta$, with $g_i$ and $h_i$ their smallest exponents, $u\in I_i$, $v\in J_i$, $0\in I_i\cap J_i$ and $I_i\oplus J_i=\{0,\dots,n-1\}$, with each element represented exactly once. Some $P_i$ contains the exponent $g_i+\beta$: if none did, every factor would have its second exponent at distance at least $2\beta$ from the smallest, and so would the product; but $\Phi(\cP_n(\beta))$ contains the exponent next to the smallest, at distance $\beta$. We fix this $i$ and put $I=I_i$ and $J=J_i$, so that $1\in I$.

\emph{Step 3: de~Bruijn's lemma.} Let $r$ be the smallest positive integer not in $I$; since $1\in I$, $r\ge2$. By Lemma~\ref{lem:debruijn}, $I$ is a union of complete blocks $\{rq,\dots,rq+r-1\}$ and $r\mid n$. In words: the factor containing the exponent $g_i+\beta$ begins with a complete block of digits $0,\dots,r-1$.

\emph{Step 4: induction.} Grouping by blocks, $P_i=\Phi(\cP_r(\beta))\,P'$, where $P'$ has coefficients $0$ or $1$ and is symmetric, because $P_i$ and $\Phi(\cP_r(\beta))$ are and $\ZR$ is a domain. Since the fixed points are determined by $c_0$ (next paragraph), $\cG_i\cong\cP_r(\beta)\times\cG'$. Irreducibility forces $\cG'$ to be a unit, so $\cG_i\cong\cP_r(\beta)$, and $r$ is prime by (2). By (1), $\Phi(\cP_n(\beta))=\Phi(\cP_r(\beta))\,\Phi(\cP_{n/r}(r\beta))$; cancelling $\Phi(\cP_r(\beta))$ in the domain $\ZR$, the product of the other factors has the $\Phi$ of $\cP_{n/r}(r\beta)$ and, since its coefficients are $0$ or $1$, also the same number of fixed points. So $\prod_{l\ne i}\cG_l\cong\cP_{n/r}(r\beta)$ in the reduced monoid, and induction on $n$ gives the stated form.

Fixed points play no role: for a symmetric element with coefficients in $\{0,1\}$ we have $c_0\in\{0,1\}$, and the conditions $f\le c_0$, $f\equiv c_0$ force $f=c_0$. Conversely, (1) produces each factorization in the list. Distinct orderings give distinct multisets of factors, because $\cP_p(\gamma)$ determines $p$ (the number of sites) and $\gamma$ (its largest exponent is $\tfrac{p-1}2\gamma$), and the scales $p_1\cdots p_{i-1}\beta$ increase strictly, so that the factors, ordered by scale, recover the ordering.
\end{proof}

In the light of this family, the theorem allows us to reread several results of Section~\ref{sec:fact}.

\begin{remark}
Theorem~\ref{thm:nonunique}(1) is the case $n=6$, $\beta=2$: $\cO(1)\times\cT(4)=\cP_2(2)\times\cP_3(4)$ and $\cO(3)\times\cT(2)=\cP_3(2)\times\cP_2(6)$ are the two orderings of $6=2\cdot3$. The progression skeleton is therefore an element all of whose factorizations have the same length, inside a monoid of infinite elasticity (Theorem~\ref{thm:elast}). The counts in (3) have been confirmed by exhaustive factorization for all $n\le30$.
\end{remark}

In summary, in the progression family non-uniqueness is completely described: it comes only from the freedom to order the digits of the mixed radix, and the length of factorizations is an invariant, $\Omega(n)$. In Section~\ref{sec:burnside} we will see that this family is also the one that connects the theory with its probabilistic origin: the progression skeleton is the skeleton of the gambler's ruin (Remark~\ref{rem:ruin}).

\section{The Burnside viewpoint and the dynamical origin}\label{sec:burnside}

In this section we first prove that skeletons of fixed degree are exactly the finite $C_2$-sets over the monoid $(\mathbb R,+)$, so that Sections~\ref{sec:setting} and~\ref{sec:cancel} are, to a large extent, the Burnside theory of $C_2$ carried over to this setting. This makes clear what is known and what is new: the additive and ring structures follow the classical model, while the multiplicative arithmetic of Section~\ref{sec:fact} has, as far as we know, no precedent. Second, we recall that skeletons are the combinatorial part of spectral duality structures, and we show that two numbers that are not multiplicative, the total number of orbits and the number of free orbits, carry their dynamical information. We compute how they behave under the product and explain what they mean.

In this section, in accordance with the convention of Section~\ref{sec:setting}, ``orbit'' refers to all $C_2$-orbits, and we say ``free orbit'' when we want to exclude the fixed points.

\subsection{Skeletons as \texorpdfstring{$C_2$}{C2}-sets}\label{subsec:burnside}

Theorem~\ref{thm:semiring}, which identifies classes of skeletons with triples $(K,\Phi,f)$, has a classical interpretation. A skeleton is, at bottom, a finite set with an involution, that is, with an action of the group $C_2$, in which each site carries a real number, $\log r_z$, that changes sign under the involution. This is exactly a $C_2$-set over the real line.

Let $C_2=\{1,\tau\}$ act on $\mathbb R$ by $\tau t=-t$. Then $(\mathbb R,+)$ is a $C_2$-monoid, that is, a monoid on which $C_2$ acts by automorphisms. A \emph{finite $C_2$-set over $\mathbb R$} is a finite $C_2$-set $X$, possibly empty, with an equivariant map $\ell:X\to\mathbb R$; morphisms are equivariant maps compatible with the maps to $\mathbb R$. These objects are added by disjoint union and multiplied through the monoid: $X\times Y$ carries the diagonal action and $\ell(x,y)=\ell(x)+\ell(y)$, which is again equivariant.

\begin{proposition}[Skeletons are $C_2$-sets over $\mathbb R$]\label{prop:burnside}
For each $K>0$, the assignment $(Z,\sigma,\kappa)\mapsto(Z,\sigma,\ell)$, with $\ell(z)=\log r_z$, is an isomorphism from the category of skeletons of degree $K$, including the empty one $\emptyset_K$, with the morphisms of \cite[Definition~5.1]{PaperVI}, onto the category of finite $C_2$-sets over $\mathbb R$. It carries the disjoint union to the disjoint union, and the product of skeletons of degrees $K_1,K_2$ to the product of $C_2$-sets over $\mathbb R$.
\end{proposition}

\begin{proof}
The involution $\sigma$ is a $C_2$-action, and $\ell$ is equivariant because $r_{\sigma z}=r_z^{-1}$. Conversely, $\kappa$ is recovered as $\kappa(z)=\sqrt K\,e^{-\ell(z)}$ \cite[Proposition~3.1]{PaperVI}, which satisfies $\kappa(z)\kappa(\sigma z)=K$ for every equivariant $\ell$. A map compatible with the involutions preserves $\kappa$ if and only if it preserves $\ell$, by the same formula. The product statement is $r_{(z,w)}=r_zr_w$, and the union statement is clear.
\end{proof}

In Burnside theory, finite $G$-sets are studied through their \emph{marks}: how many points each subgroup fixes. Here we must also keep track of the point of $\mathbb R$ over which each site lies. For a subgroup $H\le C_2$ and a point $a$ of $\mathbb R$ fixed by $H$, the mark $\varphi_{H,a}(X)$ counts the points of $X$ fixed by $H$ and lying over $a$. The marks for $H=1$, collected as $\sum_a\varphi_{1,a}(X)\,x^a$, give $\Phi$. Since $\mathbb R^{C_2}=\{0\}$, the only mark for $H=C_2$ is $\varphi_{C_2,0}=f$. The invariants of Section~\ref{sec:setting} are thus exactly the marks, and Theorem~\ref{thm:semiring} is the injectivity of the marks on isomorphism classes, together with a description of their image.

The conditions defining $\Sy^\sharp$ and the phenomena of Section~\ref{sec:cancel} also have their classical counterparts. The parity condition $f\equiv c_0(\Phi)$ (mod~$2$) is the counterpart, over the point $0$, of the congruence $\varphi_1\equiv\varphi_{C_2}$ (mod~$2$) that characterizes the image of the mark homomorphism of the Burnside ring $A(C_2)$ \cite{Dress1969,tomDieck1979}. Likewise, the failure of cancellation in Theorem~\ref{thm:nocancel} and the zero divisors of Theorem~\ref{thm:ring} are the graded form of the classical relation $[C_2]\cdot\bigl([C_2]-2\bigr)=0$ in $A(C_2)$ \cite{Solomon1967}. Indeed, over the point $0$, $\cF_1$ is the orbit $C_2/C_2$, $\cO(0)$ is the free orbit $C_2$, and $\cF_2-\cO(0)=2-[C_2]$.

The general framework in which this study fits is that of crossed Burnside rings. Oda and Yoshida \cite{OdaYoshida} call a finite $G$-set with a $G$-map, its \emph{weight function}, into a finite $G$-monoid $S$ a \emph{crossed $G$-set over $S$}; and they call the Grothendieck ring of crossed $G$-sets, with respect to disjoint union and tensor product, the \emph{crossed Burnside ring} of $(G,S)$. These rings also arise from the Dress construction applied to the Burnside functor \cite{Bouc2003}. In this language, a skeleton of fixed degree is a crossed $C_2$-set over the $C_2$-monoid $(\mathbb R,+)$, with weight function $\ell=\log r$, and the semiring of Section~\ref{sec:setting} is an analogue, for an infinite monoid, of the crossed Burnside semiring.

The scope of this comparison should be made precise. The constructions of Oda and Yoshida and of Dress, made for a finite monoid, are the model for this viewpoint, not a source of results. Since each object involves only finitely many values, the fact that the monoid is infinite causes no difficulty, and the facts we need are proved directly. Sections~\ref{sec:setting} and~\ref{sec:cancel} thus rest on known facts: from Burnside rings and, for the multiplicative group, from Pólya's theorem. By factorization we mean factorization into irreducibles in the multiplicative monoid of classes, in the sense of \cite{GHK}. We are not aware of studies of this aspect for semirings of Burnside type: their literature deals with the Burnside ring (its prime spectrum, its idempotents and its units), and not with the arithmetic of the semiring.

\subsection{Orbit counting}

The marks, $\Phi$ and $f$, are multiplicative, and this is why they have been enough to study factorization. But there are two numbers attached to a skeleton that are not marks, although they are computed from them: the number $\nu$ of $C_2$-orbits, fixed points included, and the number $k$ of free orbits, which are the orbits of the previous sections. These two numbers are not multiplicative, and they are precisely the ones that carry the dynamical content of the structures, as we will see in the next subsection. It is therefore useful to know how to compute them for a product.

\begin{proposition}[Orbit counting]\label{prop:orbits}
\begin{enumerate}[nosep]
\item For every skeleton, $\nu=(m+f)/2$ and $k=(m-f)/2$.
\item For a product $\cE=\cE_1\times\cE_2$,
\[
\nu=\nu_1\nu_2+k_1k_2=\frac{m_1m_2+f_1f_2}2,\qquad
k=2k_1k_2+k_1f_2+f_1k_2=\frac{m_1m_2-f_1f_2}2 .
\]
\item The product of two orbits is a union of two orbits if both are free, and a single orbit otherwise.
\end{enumerate}
\end{proposition}

\begin{proof}
(1) This is the orbit-counting lemma of Cauchy, Frobenius and Burnside, $\nu=\frac12\bigl(|X|+|X^{C_2}|\bigr)$, together with $\nu=k+f$; in terms of the marks $m$ and $f$, it simply says that $m=2k+f$.

(3) If $O_1=\{z,\sigma z\}$ and $O_2=\{w,\sigma w\}$ are free, $O_1\times O_2$ consists of the two free orbits $\{(z,w),(\sigma z,\sigma w)\}$ and $\{(z,\sigma w),(\sigma z,w)\}$. If one of them is a fixed point, the product is a single orbit, free unless both are fixed.

(2) This follows from (3) by counting, or from (1) and the multiplicativity of $m$ and $f$.
\end{proof}

Point (3) is the combinatorial version of the splitting of orbits in Remark~\ref{rem:splitting}: two free orbits give two free orbits, whose ratios are the sum and the difference.

\subsection{The dynamical origin}

We now recall where skeletons come from and what the numbers $\nu$ and $k$ mean \cite{PaperI,PaperV,PaperVI,PaperVII}. The paper does not depend on this interpretation, but it explains why these structures are worth studying.

As recalled in the Introduction, spectral duality structures explain why certain initial distributions make the absorption probabilities of a process with resetting independent of the resetting mechanism, with the universal value $C^\ast=1/(1+\sqrt K)$ \cite{PaperI,PaperV}. A complete SDS consists of a skeleton and spectral amplitudes; the skeleton is its combinatorial part.

The distributions that produce this independence have a simple description in terms of the skeleton. A distribution $\pi$ on $Z$ is \emph{neutral} if $\pi_z=r_z\,\pi_{\sigma z}$ for all $z$. Since $r_z>0$, on each orbit $\pi$ either vanishes at both sites or is positive at both, and in the latter case $\pi_z/\pi_{\sigma z}=r_z$. For a complete SDS, the neutral distributions are exactly those that do not depend on the spectral parameter, and they give the value $C^\ast$. Each orbit contributes an elementary neutral distribution: for each orbit $O$, let $e_O$ be the neutral distribution supported on $O$, that is, $e_O=\delta_z$ at a fixed point, and $e_O\propto r_z^{1/2}\delta_z+r_z^{-1/2}\delta_{\sigma z}$ on a free orbit. The neutral distributions are the mixtures of these: they form a simplex $\overline\Sigma$, the \emph{separatrix}, with vertices the $e_O$ and dimension $\nu-1$. On the other hand, the \emph{response rank} of a complete SDS with linearly independent spectral functions is $k$, the number of free orbits \cite[Corollary~8.9]{PaperVI}. Thus $\nu$ measures how many neutral distributions there are, and $k$ how many independent directions of response the structure has.

Among all neutral distributions there is one that stands out naturally, and it connects this interpretation with the characters of Section~\ref{sec:cancel}.

\begin{proposition}[The canonical neutral distribution]\label{prop:canonical}
Among the distributions $\pi_z\propto r_z^{\lambda}$, exactly one is neutral, the one with $\lambda=\tfrac12$ (if $r_z=1$ for every $z$, all values of $\lambda$ give this same distribution, the uniform one):
\[
\pi^{\mathrm c}_z=\frac{r_z^{1/2}}{\chi_{1/2}(\cE)}\ \propto\ \kappa(z)^{-1/2}.
\]
Its normalizing constant is the character $\chi_{1/2}$, and $\pi^{\mathrm c}(\cE_1\times\cE_2)=\pi^{\mathrm c}(\cE_1)\otimes\pi^{\mathrm c}(\cE_2)$. In terms of the vertices, $\pi^{\mathrm c}\propto\sum_O\tilde e_O$, where $\tilde e_O=\delta_z$ at a fixed point and $\tilde e_O=r_z^{1/2}\delta_z+r_z^{-1/2}\delta_{\sigma z}$ on a free orbit.
\end{proposition}

\begin{proof}
The quotient $r_z^\lambda/r_{\sigma z}^{\lambda}=r_z^{2\lambda}$ equals $r_z$ for all $z$ if and only if $\lambda=\tfrac12$ (when some $r_z\ne1$; otherwise every $\lambda$ works and all choices coincide). Since $\kappa(z)=\sqrt K/r_z$, we have $r_z^{1/2}\propto\kappa(z)^{-1/2}$. The statement on the tensor product follows from $r_{(z,w)}=r_zr_w$ and the multiplicativity of $\chi_{1/2}$.
\end{proof}

For the progression skeleton, $\pi^{\mathrm c}_j\propto e^{(j-\frac{n-1}2)\beta/2}$; for the gambler's ruin, it is a neutral initial distribution, so the ruin probability equals $C^\ast$ at every resetting rate (Remark~\ref{rem:ruin}).

We can now answer the natural question of what happens to the dynamical information when two skeletons are multiplied: how the universal value, the separatrix and the response rank change, and which part of the separatrix of the product is obtained from those of the factors.

\begin{proposition}[Dynamical invariants of a product]\label{prop:dyn}
Let $\cE=\cE_1\times\cE_2$.
\begin{enumerate}[nosep]
\item $\sqrt K=\sqrt{K_1}\sqrt{K_2}$ is multiplicative, so $C^\ast(\cE)=1/(1+\sqrt{K_1K_2})$.
\item $\dim\overline\Sigma(\cE)=\nu_1\nu_2+k_1k_2-1$, and every complete SDS on $\cE$ with linearly independent spectral functions has response rank $2k_1k_2+k_1f_2+f_1k_2$.
\item If $\pi_1$ and $\pi_2$ are neutral, then $\pi_1\otimes\pi_2$ is neutral for $\cE$. The image of $\overline\Sigma(\cE_1)\times\overline\Sigma(\cE_2)$ under $(\pi_1,\pi_2)\mapsto\pi_1\otimes\pi_2$ has codimension
\[
(\nu_1-1)(\nu_2-1)+k_1k_2
\]
in $\overline\Sigma(\cE)$ (the image is not linear, and by codimension we mean the difference of the dimensions). It is therefore all of $\overline\Sigma(\cE)$ exactly when one factor is a unit, or the factors are an orbit $\cO(t)$ and a set of fixed points $\cF_n$.
\end{enumerate}
\end{proposition}

\begin{proof}
(1) This is the multiplicativity of the degree.

(2) This is Proposition~\ref{prop:orbits}(2) together with the facts recalled above; the response rank depends only on the skeleton.

(3) Since $r_{(z,w)}=r_zr_w$ and $\sigma(z,w)=(\sigma z,\sigma w)$, the tensor product of neutral distributions is neutral. More precisely, $e_{O_1}\otimes e_{O_2}$ lies in the face of $\overline\Sigma(\cE)$ spanned by the orbits of $O_1\times O_2$ (Proposition~\ref{prop:orbits}(3)). The image of the tensor map has dimension $(\nu_1-1)+(\nu_2-1)$, since $\pi_1$ and $\pi_2$ are recovered as the marginals of $\pi_1\otimes\pi_2$, and
\[
(\nu_1\nu_2+k_1k_2-1)-(\nu_1-1)-(\nu_2-1)=(\nu_1-1)(\nu_2-1)+k_1k_2 .
\]
Both summands are nonnegative, so the codimension vanishes exactly when $\nu_1=1$ or $\nu_2=1$, and moreover $k_1k_2=0$. A skeleton with $\nu=1$ is a unit ($k=0$) or a single orbit $\cO(t)$ ($k=1$); in the second case, $k_1k_2=0$ forces the other factor to have $k=0$, that is, to be some $\cF_n$. In both cases the image is indeed the whole simplex: if one factor is a unit this is immediate, and for $\cO(t)\times\cF_n$ the map $\pi_2\mapsto e_O\otimes\pi_2$ is linear and carries the vertices $\delta_w$ of $\overline\Sigma(\cF_n)$ to the vertices $e_{O\times\{w\}}$ of $\overline\Sigma(\cE)$.
\end{proof}

Point (3) says that, in general, the product has many more neutral distributions than those obtained by combining the ones of the factors: the dynamics of the product does not reduce to that of its pieces. This is also the reason why the paper works with skeletons and not with complete structures, as the following remark explains.

\begin{remark}[Why skeletons and not complete structures]\label{rem:complete}
A complete SDS adds to the skeleton spectral amplitudes $A_j$, $j=1,\dots,N$, satisfying the full rank condition (B1) and the duality (B2) of \cite{PaperVII}. For the product with a second structure, with amplitudes $A'_{j'}$, $j'=1,\dots,N'$, the amplitudes $A_{(j,j')}(z,w)=A_j(z)A'_{j'}(w)$ again satisfy both conditions: (B1), because Kronecker products of two linearly independent families are linearly independent, and (B2), because the weights multiply. This extends the observation in \cite[Remark~4.6]{PaperVI}, where the full rank condition is said to be inherited in standard cases: the product of two complete SDS is always a complete SDS. What is not preserved is the linear independence of the spectral functions, which is the hypothesis of the rank theorem \cite[Theorem~8.8]{PaperVI}: the spectral functions of the product are the products of those of the two factors. In the known realizations, as functions of the spectral parameter $\gamma$ and up to a common nonvanishing factor, the spectral function of index $j$ is $g_{\lambda_j}(u)$, with $g_\lambda(u)=u/(1-\lambda u)$ and $u=1-\gamma$, and that of index $j'$ of the second factor is $g_{\lambda'_{j'}}(u)$. For $\lambda\ne\mu$ one has the identity
\[
g_\lambda\,g_\mu=\frac{g_\lambda-g_\mu}{\lambda-\mu},
\]
valid for all real $\lambda,\mu$, including $0$. Hence, when no eigenvalue of the first factor equals an eigenvalue of the second, the $NN'$ products lie in the span of the $N+N'$ functions $g_{\lambda_j}$, $g_{\lambda'_{j'}}$, and their linear independence fails as soon as $NN'>N+N'$, that is, $(N-1)(N'-1)>1$; each common eigenvalue $\lambda$ adds at most the function $g_\lambda^2$ to this span. So the product of two spectral realizations is not, in general, a realization to which the rank theorem applies.

What does survive is the following: \emph{every} complete SDS on the product skeleton whose spectral functions are linearly independent has the response rank predicted by the skeleton, because in that case the response rank depends only on the skeleton (\cite[Corollary~8.9]{PaperVI}; see also \cite{PaperVII}). This is Proposition~\ref{prop:dyn}(2). The skeleton is therefore the object that carries the product, and the right level at which to study it. Proposition~\ref{prop:dyn}(2) has also been verified numerically, with skeletons given by exact data and at high precision, because in double precision the singular values of the response matrix decay geometrically and the rank becomes numerically ambiguous for large $m$.
\end{remark}

The dynamical interpretation also allows us to reread two earlier results: the failure of cancellation (Theorem~\ref{thm:nocancel}) and the factorization of progression skeletons (Theorem~\ref{thm:ruin}).

\begin{remark}[Dynamical reading of the failure of cancellation]\label{rem:nocancel-dyn}
The two skeletons that Theorem~\ref{thm:nocancel} identifies after multiplication are dynamically different. For a complete SDS with linearly independent spectral functions, $\cF_2$ has response rank $0$ and a separatrix of dimension $1$, while $\cO(0)$ has response rank $1$ and a separatrix reduced to a point: two self-dual sites versus a pair of exchanged channels with equal weights. Coupled to a structure without fixed points, with $m_1$ sites say, the difference disappears, since both products have $k=m_1$ free orbits and the same ratios. In this sense, the failure of cancellation says that a factor without fixed points cannot distinguish a self-dual site from half of a balanced pair.
\end{remark}

\begin{remark}[The probabilistic origin]\label{rem:ruin}
The progression skeleton $\cP_{a-1}(\beta)$ is the skeleton of the biased random walk on $\{0,\dots,a\}$, absorbed at both ends, which jumps to the right with probability $p$ and to the left with probability $q=1-p$, and $\beta=|\log(p/q)|$; its $a-1$ sites are the interior states. With the convention of \cite{PaperI}, where ruin is absorption at $0$, its degree is $K=(p/q)^a$, and $C^\ast=1/(1+\sqrt K)=1/(1+(p/q)^{a/2})$ is the classical ruin probability from the midpoint (for even $a$). With the initial distribution $\pi^{\mathrm c}$ of Proposition~\ref{prop:canonical}, the ruin probability equals $C^\ast$ at every resetting rate. This invariance, first observed for the walk started at the midpoint (for even $a$, the vertex of the separatrix at the fixed point), is the origin of the theory of spectral duality structures; the other absorption channel corresponds to $K^{-1}$.

In this reading, Theorem~\ref{thm:ruin}(1) says the following: the skeleton of the biased walk with $rs$ interior states and bias $\beta$ is the product of the skeletons of two walks of the same kind, one with $r$ interior states and the same bias $\beta$, and one with $s$ interior states and bias $r\beta$, that is, with ratio $(p/q)^r$, as if each of its steps grouped $r$ steps of the first. Writing each state as $j=j_1+rj_2$, with $0\le j_1<r$ and $0\le j_2<s$, the first factor runs over the digit $j_1$ and the second over the digit $j_2$. This is an identity between skeletons, not between the processes. The different factorizations correspond to the different orders in which the factors of $rs$ are grouped: for instance, $6=2\cdot3$ or $6=3\cdot2$. The example of Theorem~\ref{thm:nonunique}(1) is the gambler's ruin on $\{0,\dots,7\}$, with its two factorizations $\cP_6(2)\cong\cP_2(2)\times\cP_3(4)\cong\cP_3(2)\times\cP_2(6)$.
\end{remark}

As noted above, the identities of this section are identities between skeletons, not between processes. Which skeletons come from a process is the realizability problem, to which we return in Section~\ref{sec:open}.

\section{Open problems}\label{sec:open}

The results of this paper describe the arithmetic of the monoid in three situations: skeletons without fixed points and with an admissible number of orbits, where the criterion by centres is exact; generic skeletons, which are irreducible; and the progression family, whose factorizations are all listed. Outside these situations natural questions remain, which we group into four blocks: reducibility in general (items 1, 2 and 6), the arithmetic invariants of the monoid (items 3 and 4), the multiplicative group (item 5) and the relation with the dynamics (item 7).

\begin{enumerate}
\item \emph{General number of orbits.} For $f=0$ and non-admissible $k$ there are factorizations without an orbit factor (Theorem~\ref{thm:orbitfactor}(b)), and the criterion by centres does not apply. Proposition~\ref{prop:general} reduces reducibility to a finite search, but an explicit criterion is a question about sum and difference sets. The simplest example, $k=9$ (Example~\ref{ex:k9}), shows that the obstruction already appears with three copies of an irreducible with three orbits.

\item \emph{Skeletons with $f>0$.} We know that $\Phi$ and $f$ must be factored simultaneously, and the difficulty is that the two components are not independent: parity links $f$ with the constant term of $\Phi$, and a factor with $f=1$ necessarily has an odd number of sites (Proposition~\ref{prop:parity}).

\item \emph{The submonoid $f=1$.} It is not half-factorial (Theorem~\ref{thm:f1}). Is its elasticity finite? The question makes sense because the mechanism that gives infinite elasticity in the full monoid (Theorem~\ref{thm:elast}) does not act in this submonoid: that mechanism comes from the failure of cancellation and needs factors with $f\ne1$. The different lengths in Theorem~\ref{thm:f1} have another origin: they come only from the coefficients of $\Phi$, through the splitting of exponents. We do not know whether this second mechanism is enough, by itself, to make the elasticity unbounded.

\item \emph{Length sets.} Is the length set of $\cO(t/2)^{2j+1}$ the full interval $\{2,\dots,2j+1\}$ for every $j$ (it is for $j\le3$)? Are all length sets intervals, or is at least the set of distances finite? These are the usual questions of the theory of non-unique factorization \cite{GHK}, and here they remain open.

\item \emph{The multiplicative group in general} (Problem~\ref{prob:gr}). For commensurable ratios, Pólya's theorem describes the group (Theorem~\ref{thm:polya}); for incommensurable ratios, on the other hand, the question is of Pólya type in several variables.

\item \emph{Irreducibles.} It would be interesting to describe them in general. We know that non-decomposable types and $k=1$ give irreducibles (Proposition~\ref{prop:types}), that $f=2$ with $m$ a power of $2$ implies irreducibility (Lemma~\ref{lem:f2irr}), that irreducibility is generic in every type other than $(4,0)$ and $(n,n)$ (Theorem~\ref{thm:generic}), and that Theorem~\ref{thm:orbitfactor} gives an exact criterion for $f=0$ and admissible $k$. But what is the structure of the reducible locus beyond these cases?

\item \emph{Realizability.} Which structures are realized by positive Markov processes? Do the realizable classes form a subsemiring? These questions belong to the line of work on realizability and are not treated in this paper. We do observe that, in the progression family, the product of the skeleton of a fine walk and that of a coarse walk is again the skeleton of a walk (Theorem~\ref{thm:ruin}(1) and Remark~\ref{rem:ruin}), so that some products of realizable skeletons are again realizable.
\end{enumerate}

\section*{Acknowledgements}

The author thanks Tony Newton for stimulating discussions throughout this research program. He is also grateful to several colleagues for their comments on this and related work, which helped to improve it, and to the Mathlib community for developing the library on which the formal verification is built.

\section*{Use of artificial intelligence tools and formal verification}

The author used AI tools (Claude, Anthropic) in every phase of the research: in the initial exploration, in the numerical and formal verification of the results (writing the code for the computations and for the Lean verification), in the editing of the Spanish text and in the writing of the English text. The direction of the research and the mathematical content are the author's, who takes full responsibility for them.

The main results of the paper have been verified in the proof assistant Lean~4 \cite{Lean4}, with the Mathlib library \cite{Mathlib}. The code was written by Claude under the author's direction, and the author compiled it on his own computer with Lean~4 (version \texttt{v4.36.0-rc1}) and Mathlib (tag \texttt{v4.36.0-rc1}, revision \texttt{625d891}); it is available at \href{https://doi.org/10.5281/zenodo.23883564}{doi:10.5281/zenodo.23883564}. The following are verified:
\begin{itemize}[nosep]
\item the classification of skeletons and the identification with $\mathbb R_{>0}\times\Sy^\sharp$ for non-empty skeletons (Theorems~\ref{thm:class} and~\ref{thm:semiring}), the multiplicativity of the invariants (Proposition~\ref{prop:mult}) and Lemma~\ref{lem:closed};
\item the failure of cancellation (Theorem~\ref{thm:nocancel}), the determination by characters (Proposition~\ref{prop:chi}), the characterization of equality in $K_0$ (Theorem~\ref{thm:K0}), the Grothendieck ring except for its graded form (Theorem~\ref{thm:ring}), and Theorem~\ref{thm:polya}, the latter \emph{assuming} Pólya's theorem in one variable, which enters as an explicit hypothesis and is not proved;
\item Proposition~\ref{prop:atomic}, the non-uniqueness of Theorem~\ref{thm:nonunique}(1), Theorems~\ref{thm:halffact}, \ref{thm:elast} and~\ref{thm:f1}, Lemma~\ref{lem:f2irr} and the statements of Example~\ref{ex:k9} about $\cF_3\times C$, except for the irreducibility of $C$;
\item the main implication of item (a) of Theorem~\ref{thm:orbitfactor} and its item (b), the first part of Theorem~\ref{thm:generic}, and Theorem~\ref{thm:ruin}, with Lemma~\ref{lem:debruijn}, except for the multinomial count in item~(3).
\end{itemize}
The second part of Theorem~\ref{thm:generic}, Proposition~\ref{prop:general}, Lemma~\ref{lem:centre}(3), the reformulations (1) and (2) of Theorem~\ref{thm:orbitfactor}(a) and Corollary~\ref{cor:small}, the graded form of Theorem~\ref{thm:ring}, the isomorphism for $K_0$, the multinomial count and Section~\ref{sec:burnside} have not been formalized, nor have the other elementary results not listed above. For all the verified results, the proof given in the text is the one verified in Lean, except for Theorem~\ref{thm:orbitfactor}(b), where Lean uses explicit witnesses (ratios that are powers of $4$) instead of ratios linearly independent over $\mathbb Q$.

In the code, the structure \texttt{Esqueleto} is that of non-empty skeletons. The empty skeleton corresponds to the zero $(0,0)$ of \texttt{Sharp0}, and its role in Theorems~\ref{thm:class} and~\ref{thm:semiring} is immediate.

Lean checks the statements as they are written in the code. The correspondence between these formal statements and those of the paper has been checked by the author, who has also checked all the statements and proofs of the text and takes full responsibility for its content.

\section*{Statements and declarations}

\paragraph{Funding.} The author received no specific funding for this work.

\paragraph{Competing interests.} The author has no financial or non-financial interests to disclose.

\paragraph{Data and code availability.} No datasets were generated or analysed in this study. The code of the formal verification in Lean is available at \href{https://doi.org/10.5281/zenodo.23883564}{doi:10.5281/zenodo.23883564}. The other computational checks mentioned in the text are elementary and can be reproduced from the descriptions given; the scripts are available from the author on request.

\end{document}